\documentclass{amsart}

\usepackage{amsmath}
\usepackage{amssymb}
\usepackage{rotating}
\usepackage{graphicx}
\usepackage[tableposition=below]{caption}
\usepackage{bm}

\DeclareGraphicsExtensions{.png,.pdf,.eps}
\usepackage[all,2cell,cmtip]{xy}
\usepackage{tikz}
\usepackage{color}
\usepackage{longtable}
\usepackage{soul}

\usepackage{here}

\newtheorem{theorem}{Theorem}[section]
\newtheorem{lemma}[theorem]{Lemma}
\newtheorem{corollary}[theorem]{Corollary}
\newtheorem{proposition}[theorem]{Proposition}

\newtheorem{claim}[theorem]{Claim}

\theoremstyle{definition}

\newtheorem{definition}[theorem]{Definition}

\newtheorem{remark}[theorem]{Remark}
\newtheorem{assumption}[theorem]{Assumption}

\newtheorem{example}[theorem]{Example}

\def\P{{\mathbb P}}
\def\Q{{\mathbb Q}}
\def\R{{\mathbb R}}
\def\Z{{\mathbb Z}}

\def\cE{{\mathcal E}}

\def\cM{{\mathcal M}}
\def\cN{{\mathcal{N}}}
\def\cO{{\mathcal{O}}}

\def\cS{{\mathcal S}}

\def\cU{{\mathcal U}}

\def\Q{{\mathbb{Q}}}

\def\rat{\dashrightarrow}

\def\cOperatorname#1{\mathop{\rm #1}\nolimits}

\def\Chow{\cOperatorname{Chow}}

\def\codim{\cOperatorname{codim}}

\def\deg{\cOperatorname{deg}}

\def\rat{\cOperatorname{RatCurves}}

\def\NE{{\cOperatorname{NE}}}
\def\ME{{\cOperatorname{ME}}}

\newcommand{\cME}[1]{\cOverline{\ME}}

\renewcommand{\labelenumi}{(\roman{enumi})}

\makeindex

\begin{document}
%\pagewiselinenumbers

\title{Fano varieties with large pseudoindex and non-free rational curves}
%\title[{Contractions and non-free rational curves on Fano varieties}]{Extremal contractions and non-free rational curves on Fano varieties}
%Contractions and the minimal anticanonical degree of non-free rational curves on Fano varieties

\author{Kiwamu Watanabe}
\date{\today}
\address{Department of Mathematics, Faculty of Science and Engineering, Chuo University.
1-13-27 Kasuga, Bunkyo-ku, Tokyo 112-8551, Japan}
\email{watanabe@math.chuo-u.ac.jp}
\thanks{The author is partially supported by JSPS KAKENHI Grant Number 21K03170.}


\begin{abstract} For $n\geq 4$, let $X$ be a complex smooth Fano $n$-fold whose minimal anticanonical degree of non-free rational curves on $X$ is at least $n-2$. We classify extremal contractions of such varieties. As an application, we obtain a classification of Fano fourfolds with pseudoindex and Picard number greater than one. Combining this result with previous results, we complete the classification of smooth Fano $n$-folds with pseudoindex at least $n-2$ and Picard number greater than one. This can be seen as a generalization of various previous results. We also discuss the relations between pseudoindex and other invariants of Fano varieties.
\end{abstract}

\maketitle

\section{Introduction} 
\subsection{Classification of Fano varieties from the perspective of index and pseudoindex} A complex smooth projective variety $X$ is called {\it Fano} if its anticanonical divisor $-K_X$ is ample. From both classical and modern perspectives on projective geometry and the Minimal Model Program, Fano varieties carry substantial significance. It is known that there are finitely many families of smooth Fano varieties for each dimension \cite{KMM92}; in particular, Fano threefolds were classified by G. Fano, V. A. Iskovskih, V. V. Shokurov, T. Fujita, S. Mori, S. Mukai and K. Takeuchi (see \cite{Isk} and references therein). The classification becomes more challenging in dimensions four and higher; many authors have investigated Fano fourfolds (see some examples \cite{Wis89, Wis90, Prok94, Lan98, Tsu10-2, Tsu12, Casa13, Casa13-2, Fuj14, Casa17, Casa20, CasaRom22, Casa22,  HLM22}), but they are still far from a complete classification.

For a Fano $n$-fold $X$ with $n \geq 4$, we define the {\it Fano index} $i_X$ and the {\it pseudoindex} $\iota_X$ of $X$ by
$$
i_X:=\max \left\{m\in \Z\,\, \middle|\,\, -K_X=mL~\mbox{with}~L\in {\rm Pic}(X)\right\}
$$ 
and 
$$
\iota_X:=\min \{(-K_X)\cdot C\mid C\subset X ~\mbox{is a rational curve}\}
$$ respectively. Remark that $i_X$ divides $\iota_X$. By Kobayashi-Ochiai Theorem \cite{KO}, $i_X\leq n+1$; moreover if $i_X= n+1$, then $X$ is isomorphic to a projective space $\P^n$, and if $i_X=n$, then $X$ is isomorphic to a smooth quadric $Q^n$. If $i_X= n-1$, then $X$ is called a {\it del Pezzo variety}, which is completely classified by T. Fujita \cite{Fuj1, Fuj2}; if $i_X= n-2$, then $X$ is called a {\it Mukai variety}, which is completely classified by S. Mukai \cite{Muk89}.  On the other hand, as a generalization of Kobayashi-Ochiai Theorem \cite{KO}, ${\iota}_X\leq \dim X+1$; moreover if ${\iota}_X= n+1$, then $X$ is isomorphic $\P^n$ \cite{CMSB}, and if ${\iota}_X= n$, then $X$ is isomorphic $Q^n$ \cite{Mi, DH17}. Different from the case of the Fano index, a complete classification of Fano varieties with pseudoindex $n-1$ and $n-2$ is not known. On the other hand, according to the results of J. A. Wi\'sniewski \cite{Wis90b}, G. Occhetta \cite[Corollary~4.3]{Occ06}, and the author \cite{Wat24}, the classification of smooth Fano $n$-folds with Picard number greater than one and pseudoindex greater than $\frac{n+1}{2}$ is known. From these results, the classification of smooth Fano $n$-folds with pseudoindex at least $n-2$ and Picard number greater than one is reduced to the four-dimensional case (see Theorem~\ref{them:3''}).

\subsection{Main results} 
To classify Fano fourfolds of pseudoindex two, we study extremal contractions, which play a crucial role in understanding the geometry of higher-dimensional Fano varieties. To explain our results, we introduce a new invariant of Fano varieties.
Let $X$ be a smooth Fano variety. A non-trivial morphism $f: \P^1 \to X$ is called a {\it free rational curve} if $f^{\ast} T_X$ is nef, where $T_X$ is the tangent bundle of $X$. It is called a {\it non-free rational curve} otherwise. We denote by ${\rm Hom}^{\rm nf}(\P^1, X)$ the closed subscheme of the Hom scheme ${\rm Hom}(\P^1, X)$ parametrizing non-free rational curves.
\begin{definition}\label{def:tau} For a smooth Fano variety $X$, 
\begin{equation} \nonumber
  \tau_X:=
  \begin{cases}
   \min\bigl\{\deg f^{\ast}(-K_X) \mid [f]\in {\rm Hom}^{\rm nf}(\P^1, X)\bigr\} & \text{if ${\rm Hom}^{\rm nf}(\P^1, X)\neq \emptyset$,} \\
   + \infty              & \text{otherwise.}
  \end{cases}
\end{equation}
We call $\tau_X$ the {\it free index} of $X$.
\end{definition}
%If $\wedge^r T_X$ is nef for $1\leq r< \dim X$, then we have $\tau_X\geq \dim X-r+1$ (see Lemma~\ref{lem:nonfree}).  
%[{=Theorem~\ref{them:smooth:contraction}}]

Our main results are the following:
 
\begin{theorem}\label{them:1}\rm Let $X$ be a smooth Fano variety of dimension $n \geq 4$ and $\tau_X\geq n-2$. Let $\varphi: X \to Y$ be a contraction of an extremal ray. If $\varphi$ is of fiber type, then $\varphi$ is one of the following:
 \begin{enumerate}
 \item a smooth morphism;
 \item a $\P^{n-3}$-fibration which is not equidimensional;
 \item an equidimensional $Q^{n-2}$-fibration.
 \end{enumerate} 
 Moreover, any fiber of $\varphi$ has codimension at least two {if $X$ is not a product of $\P^1$ and a variety}.
\end{theorem}

Following the idea of the proof of Yasutake's result \cite{Yas14}, we also prove the following:

\begin{theorem}\label{them:2} \rm  Let $X$ be a smooth Fano variety of dimension $n \geq 4$. Assume that $\tau_X\geq n-2$. If $X$ admits a birational contraction of an extremal ray, then $X$ is isomorphic to one of the following:
\begin{enumerate}
\item $\P(\cO_{\P^{n-1}}\oplus \cO_{\P^{n-1}}(1))$; %In this case, $\wedge^2T_X$ is nef.
\item $\P(\cO_{Q^{n-1}}\oplus \cO_{Q^{n-1}}(1))$;
\item $\P(\cO_{\P^{n-1}}\oplus \cO_{\P^{n-1}}(2))$;
\item the blow-up of $\P^n$ along a line;
\item the blow-up of $Q^4$ along a line;
\item the blow-up of $Q^4$ along a conic which is not on a plane contained in $Q^4$;
\item $\P^1\times \P(\cO_{\P^{n-2}}\oplus\cO_{\P^{n-2}}(1))$.
\end{enumerate}
\end{theorem}

For a smooth Fano fourfold $X$, $\tau_X\geq 2$ if and only if $\iota_X\geq 2$ (see Lemma~\ref{lem:tau:iota:4}). As an application of Theorem~\ref{them:1} and \ref{them:2}, we obtain a classification of Fano fourfolds with pseudoindex and Picard number greater than one (Theorem~\ref{them:3'}). Consequently, we complete a classification of Fano $n$-folds $X$ with $\iota_X\geq n-2$ and $\rho_X\geq 2$:
\begin{theorem}[{$=$Theorem~\ref{them:3''} and \ref{them:3'}}]\label{them:3} \rm  Let $X$ be a smooth Fano variety of dimension $n \geq 4$. If $\iota_X\geq n-2$ and $\rho_X\geq 2$, then $X$ is isomorphic to one of the following:
\begin{enumerate}
\item $\P^3\times \P^3$;
\item $\P(\cO_{\P^{3}}^{\oplus 2}\oplus\cO_{\P^{3}}(1))$;
\item $\P^3\times \P^2$;
\item $Q^3\times \P^2$;
\item $\P(T_{\P^3})$;
\item $\P(\cO_{\P^{3}}\oplus \cO_{\P^{3}}(1))$. %In this case, $\wedge^2T_X$ is nef.
\item $\P(\cO_{\P^{2}}^{\oplus 2}\oplus\cO_{\P^{2}}(1))$;
%\item $X$ is isomorphic to $\P^1\times \P(\cO_{\P^{2}}\oplus\cO_{\P^{2}}(-1))$.
\item a divisor on $\P^2\times \P^3$ of bidegree $(1, 1)$;
\item $\P^2\times \P^2$;
\item {$\P^1\times Q^3$; or}
\item a Fano fourfold of $i_X=2$.
\end{enumerate}
\end{theorem}

Remark that Fano fourfolds of index two and Picard number greater than one are completely classified in \cite{Wis90}.

\subsection{Positivity of the exterior power of tangent bundles} In a pioneering study, J. P. Demailly, T. Peternell, and M. Schneider \cite{DPS94} established that a smooth projective variety with a nef tangent bundle becomes a locally trivial Fano fibration over an Abelian variety up to the differences in the finite \'{e}tale covers. This result is a {\it decomposition theorem} for varieties with nef tangent bundle. Also, \cite{DPS94} and L. E. Sol\'a Conde, Wi\'sniewski \cite[Theorem~4.4]{SW04} show that all Mori contractions are smooth. This result is a {\it contraction theorem} for varieties with nef tangent bundle. After that, similar problems were considered for varieties with nef (strictly nef or ample) second exterior power of the tangent bundles \cite{CP92, CS95, Yas12, Yas14, Sch18, LOY19}. Finally, the problem was solved by the author \cite{Wat21}, including all previous results. However, a similar problem can be conceived for varieties $X$ where the $r$-th exterior power of the tangent bundle is nef for $r<\dim X$, a decomposition theorem was obtained by the author \cite{Wat} via C. Gachet's result \cite{Gac22}, which is a generalization of a decomposition theorem for varieties with nef tangent bundle \cite{DPS94} and varieties whose second exterior power of the tangent bundle is nef \cite[Corollary~1.6]{Wat21}. According to this result, the study of smooth projective varieties $X$ with nef $\wedge^{r} T_X$ $(1\leq r<n)$ is reduced to the case of Fano varieties. 

As an application of Theorem~\ref{them:1} and \ref{them:2}, we generalize all main results for varieties with nef $\wedge^2T_X$ in \cite{Wat21} to the case where $\wedge^3T_X$ is nef. More precisely, for a smooth Fano variety of dimension $n \geq 4$, the condition $\tau_X\geq n-2$ holds if $\wedge^r T_X$ is nef for $1\leq r\leq 3$ (see Lemma~\ref{lem:nonfree}). Thus Theorem~\ref{them:1} and \ref{them:2} can be regarded as a generalization of \cite[Theorem~5.2]{DPS94}, \cite[Theorem~4.4]{SW04} and \cite[Theorem~1.7]{Wat21} for Fano varieties. On the other hand, Gacet classified smooth projective varieties with strictly nef $\wedge^r T_X$ for $r\leq 4<\dim X$ \cite[Theorem~1.3, 1.4]{Gac22}. By the above decomposition theorem for varieties with nef exterior power of tangent bundle, any smooth projective variety with strictly nef $\wedge^r T_X$ for $r\leq 4<\dim X$ is Fano; thus Theorem~\ref{them:3} can be seen as a generalization of \cite[Theorem~1.3, 1.4]{Gac22} (see Corollary~\ref{cor:str:nef}).

\subsection{Outline of the paper} This paper is organized as follows. The second section defines standard notation and terminology used throughout the paper. It discusses the structure of extremal contractions and summarizes known results on Fano varieties with large pseudoindex. Additionally, it introduces the nef order $\nu_X$ and discusses its relation to pseudoindex and free index. This section also explores families of rational curves and their properties, including dominating, unsplit, and locally unsplit families.  The third section provides the proof of Theorem~\ref{them:1}, a classification of extremal contractions of Fano varieties with large free index. The fourth section demonstrates how to compute various invariants, including $\tau_X, \iota_X,  i_X, \nu_X$ for specific examples of Fano varieties that fit the classifications described in Theorem~\ref{them:2} and \ref{them:3}. The fifth section discusses Fano varieties with birational extremal contractions, which yields Theorem~\ref{them:2}. The sixth section summarizes the classification of Fano $n$-folds with $\iota_X \geq n-2$ and Picard number $\rho_X \geq 2$. It completes the classification of Fano fourfolds with pseudoindex and Picard number greater than one. This paper significantly advances the classification of higher-dimensional Fano varieties by focusing on their extremal contractions and the role of non-free rational curves. It establishes a clear framework for understanding the structure and invariants of these varieties, contributing to the broader effort of classifying all Fano varieties.

\section{Preliminaries}
\subsection*{Notation and Conventions}\label{subsec:NC} We will make use of the standard notation as in \cite{Har}, \cite{Kb}, \cite{KM}, \cite{L1} and \cite{L2}. In this paper, we work over the complex number field.
\begin{itemize}
\item For a projective variety $X$, we define $N_1(X)$ (resp. $N^1(X)$) as the vector space of one-cycles (resp. one-cocycle) on $X$ with real coefficients, considering numerical equivalence. The dimension of $N_1(X)$ is referred to as {\it the Picard number of $X$} denoted by $\rho_X$.
\item For a projective variety $X$, ${\rm Chow}(X)$ stands for the Chow scheme parametrizing cycles on $X$ (see \cite[Section~I.3]{Kb} for details).
\item By a {\it curve}, we mean a projective variety of dimension one. 
\item An {\it $F$-bundle} over a smooth projective variety $X$ is a smooth morphism $f: Y \to X$ between smooth projective varieties, with each fiber isomorphic to a variety $F$. We denote by $T_f$ the relative tangent bundle of $f: Y \to X$.
\item An {\it $F$-fibration} over a projective variety $X$ is a surjective morphism $f: Y \to X$ between projective varieties, with general fiber isomorphic to a variety $F$.
\item Given a vector bundle $E$ on a projective variety $X$, we use $\P(E)$ to represent the projectivization ${\rm Proj}({\rm Sym},E)$, and $\cO_{\P(E)}(1)$ to denote the tautological line bundle on $\P(E)$.
\item $\P^n$ denotes projective $n$-space, while $Q^n$ represents a smooth quadric hypersurface in $\P^{n+1}$.
\end{itemize}

\subsection{Extremal contractions and Fano varieties with large pseudoindex} Let $X$ be a smooth projective variety, and let $\varphi: X\to Y$ be a contraction of a $K_X$-negative extremal ray $R$, referred to as an {\it elementary contraction}. Every general fiber of $\varphi$ is a smooth Fano variety. The {\it length} of the ray $R$ is defined as 
$$ 
\ell(R):=\min \{-K_X\cdot C\mid C\subset X ~\mbox{is~a~rational~curve~such~that~} [C]\in R\}.
$$
A rational curve $C$ on $X$ is termed a {\it minimal extremal rational curve} of $R$ if $[C]\in R$ and $-K_X\cdot C=\ell(R)$. Tsukioka \cite{Tsu12} introduced the minimal length of extremal rays for a smooth Fano variety as follows:
\begin{definition}\label{def:length}\rm For a smooth Fano variety $X$, the {\it minimal length} of extremal rays for $X$ is defined as
$$
\ell_X
:=\min \{\ell(R)\mid  R~\mbox{is an extremal ray of}~\overline{ NE}(X)\}.
$$
\end{definition} 

\begin{proposition}\label{prop:div:iota}\rm Let $X$ be a smooth Fano variety of dimension $n \geq 3$ and $\rho_X=2$. Assume that $X$ admits a birational elementary contraction sending a divisor to a point. If another elementary contraction is a $\P^1$-bundle, then $\iota_X=\ell_X=2$.  
\end{proposition}

\begin{proof} This follows from the proof of \cite[Proposition~2]{Tsu12}.
\end{proof}

The structure of a contraction of a $K_X$-negative extremal ray $R$ is constrained by the following inequality, which relates the dimensions of the exceptional divisor and the fiber, as well as the length of $R$:
\begin{proposition}[{Ionescu-Wi\'sniewski inequality \cite[Theorem~0.4]{Ion86}, \cite[Theorem~1.1]{Wis91} }]\label{prop:Ion:Wis}\rm Let $X$ be a smooth projective variety, and let $\varphi: X \to Y$ be a contraction of a $K_X$-negative extremal ray $R$, with $E$ representing its exceptional locus. Consider $F$ as an irreducible component of a non-trivial fiber of $\varphi$. Then
\begin{eqnarray}\label{IW:inequal} \nonumber
\dim E + \dim F \geq  \dim X + \ell(R)- 1.
\end{eqnarray}
\end{proposition}

It is expected that smooth Fano varieties with large pseudoindex have simple structures. We collect some results on Fano varieties with large pseudoindex:
%[{\cite{CMSB}, \cite{Keb02}, \cite{DH17}, \cite{Wis90b}, \cite[Corollary~4.3]{Occ06}, \cite[Theorem~1.4]{Wat24}}]
\begin{theorem}\label{them:fano:collect}\rm Let $X$ be a smooth Fano variety of dimension $n$ with pseudoindex $\iota_X$. Then, the following hold.
\begin{enumerate}
\item If $\iota_X\geq n+1$, then $X$ is isomorphic to $\P^n$ \cite{CMSB, Keb02}. 
\item If $\iota_X= n$, then $X$ is isomorphic to $Q^n$ \cite{DH17}.
\item If $\iota_X>\dfrac{n}{2}+1$, then $\rho_X=1$ \cite{Wis90b}.
\item If $\iota_X=\dfrac{n}{2}+1$ and $\rho_X>1$, then $X$ is isomorphic to $(\P^{\iota_X-1})^2$ \cite[Corollary~4.3]{Occ06}.
\item If $\iota_X=\dfrac{n+1}{2}$ and $\rho_X>1$, then $X$ is isomorphic to one of the following:
$$\P(\cO_{\P^{\iota_X}}(2)\oplus\cO_{\P^{\iota_X}}(1)^{\oplus \iota_X-1}); \quad \P^{\iota_X-1}\times Q^{\iota_X};\quad 
\P(T_{\P^{\iota_X}});\quad \P^{\iota_X-1}\times \P^{\iota_X}$$\cite[Theorem~1.4]{Wat24}. 
\end{enumerate}
\end{theorem}

\subsection{Relations between pseudoindex, free index, and other invariants}
\begin{definition}\label{def:wedge}\rm For a smooth Fano variety $X$, we define an invariant of $X$ as 
$$
\nu_X
:=\min \{k\in \Z_{>0}\mid  \wedge^kT_X~\mbox{is nef}\}.
$$ We call this invariant $\nu_X$ the {\it nef order} of $X$.
\end{definition} 

\begin{remark}\label{rem:wedge}\rm Let $X$ be a smooth projective variety and $E$ a vector bundle of rank $r$ on $X$. According to \cite[Theorem~3.3]{LN05}, if $\wedge^k E$ is nef for some integer $k$ with $0<k<r$, then so is $\wedge^{k+1} E$. Thus, for a smooth Fano variety $X$ and any integer $k$ with $\nu_X\leq k\leq r$, the bundle $\wedge^{k}T_X$ is nef.
\end{remark}

Let us remark that the positivity of the exterior power of the tangent bundle and our new invariant $\tau_X$ are involved as follows:

\begin{lemma}[{\cite[Lemma~1.3]{CP92}, \cite[Lemma~2.9]{Yas12}, \cite[Lemma~3.2]{Wat21}}]\label{lem:nonfree}\rm Let $X$ be a smooth Fano variety of dimension $n\geq 4$. If $\wedge^r T_X$ is nef for some $1\leq r<n$, then we have $\tau_X\geq \dim X-r+1$.
\end{lemma}

\begin{remark}\label{rem:inv}\rm Let $X$ be a smooth Fano variety of dimension $n\geq 4$. By definition, we have 
\begin{enumerate}
\item $\iota_X$ is a multiple of $i_X$;
\item $\tau_X\geq \iota_X$.
\end{enumerate}
Moreover by Lemma~\ref{lem:nonfree}, $\tau_X\geq n-\nu_X+1$. These imply the following:
$$
i_X\geq n-m\Longrightarrow \iota_X\geq n-m\Longrightarrow \tau_X\geq n-m;
$$
$$
m+1\geq \nu_X\Longrightarrow \tau_X\geq n-m.
$$
\end{remark}

\subsection{Deformation theory of rational curves}
To explore the structure of a uniruled variety, it is customary to investigate families of rational curves. Let us recall the space of rational curves $\rat^n(X)$, constructed as the normalization of the subscheme of the Chow scheme $\Chow(X)$ that parametrizes rational curves on $X$. An {\it irreducible component of $\rat^n(X)$} is referred to as a {\it family of rational curves} $\cM$ on $X$. This family $\cM$ is equipped with a $\P^1$-bundle $p: \cU \to \cM$, and an evaluation morphism $q: \cU \to X$, which is described in detail in \cite[Section~II.2]{Kb}.
A rational curve parametrized by $\cM$ is called an {\it $\cM$-curve}. Since any $\cM$-curve is numerically equivalent to each other, the family $\cM$ determines a numerical class; we denote it by $[\cM] \in N_1(X)$. The {\it anticanonical degree of the family} $\cM$ refers to the intersection number  denoted as $\deg_{(-K_X)} \cM$, which is calculated as $-K_X \cdot C$, where $[C]$ represents any curve belonging to the family $\cM$. A half line generated by the class $[M] \in N_1(X)$ is denoted by $\R_{\geq 0}[\cM]$. 
The union of all $\cM$-curves is denoted by ${\rm Locus}(\cM)$. For a point $x \in X$, the normalization of $p(q^{-1}(x))$ is denoted by $\cM_x$, and we use ${\rm Locus}(\cM_x)$ to refer to the union of all $\cM_x$-curves.
\begin{definition}\label{def:dom:cov:unsplit}\rm Under the above notation:
\begin{enumerate}
\item $\cM$ is a {\it dominating family} (resp. {\it covering family}) if the evaluation morphism $q: \cU \to X$ is dominant (resp. surjective).
%\item $\cM$ is a {\it minimal rational component} if it contains a free rational curve with the minimal anticanonical degree.
\item $\cM$ is {\it locally unsplit} if, for a general point $x\in {\rm Locus}(\cM)$, $\cM_x$ is proper.
\item $\cM$ is {\it unsplit} if $\cM$ is proper.
\end{enumerate}
\end{definition}
For instance, a family of rational curves $\cM$ is locally unsplit if it is a dominating family with minimal degree concerning some ample line bundle on $X$. A family of rational curves $\cM$ is a dominating family if and only if there exists a free $\cM$-curve (see, for instance, \cite[IV Theorem~1.9]{Kb}). %A minimal rational component is a locally unsplit dominating family.
By definition, an unsplit family is locally unsplit, and an unsplit dominating family is a covering family. Mori's bend and break \cite[Theorem~4]{Mori79} tells us that a locally unsplit family of rational curves $\cM$ satisfies $\deg_{(-K_X)} \cM \leq \dim X +1$. We often use the following results: 

\begin{lemma}[{\cite[Lemma~2.7]{Wat}}]\label{lem:degeneration:curves}\rm Let $X$ be a smooth projective variety and $\cM\subset \rat^n(X)$ a family of rational curves. If $\cM$ is not proper, then there exists a rational $1$-cycle $Z=\sum_{i=1}^s a_i Z_i$ satisfying the following conditions:
\begin{enumerate}
\item $Z$ is algebraically equivalent to $\cM$-curves, where each $a_i$ is a positive integer and each $Z_i$ is a rational curve;
\item $\sum_{i=1}^s a_i\geq 2$.
\end{enumerate}
We term this rational $1$-cycle $Z=\sum_{i=1}^s a_i Z_i$ as a {\it degeneration of $\cM$-curves}.
\end{lemma}

\begin{proposition}[{\cite[IV Corollary~2.6]{Kb}}]\label{prop:Ion:Wis:2}\rm Let $X$ be a smooth projective variety and $\cM$ a locally unsplit family of rational curves on $X$. For a general point $x \in {\rm Locus}(\cM)$, 
$$\dim {\rm Locus}(\cM_x) \geq \deg_{(-K_X)}\cM+\codim_X{\rm Locus}(\cM) -1.
$$ 
Moreover, if $\cM$ is unsplit, the above inequality holds for any point $x \in {\rm Locus}(\cM)$.
\end{proposition}

For a smooth projective variety $X$, let $\cM$ be an unsplit covering family of rational curves on $X$. Two points $x_1, x_2 \in X$ are {\it $\cM$-equivalent} if connected by a connected chain of $\cM$-curves. By \cite{Cam92, KMM92} (see also \cite[Chapter~5]{DebB} and \cite[IV Theorem~4.16]{Kb}), there exists a proper surjective morphism $\varphi: X^0\to Y$ from a nonempty open subset $X^0 \subset X$ onto a normal variety $Y$ whose fibers are $\cM$-equivalent classes. We call $\varphi: X^0\to Y$ the {\it rationally chain connected fibration with respect to $\cM$} (RCC-fibration with respect to $\cM$ for short). When $X^0$ coincides with the whole space $X$, we say that $\varphi: X=X^0\to Y$ is the {\it geometric quotient for $\cM$}.

\begin{theorem}[{cf. \cite[Theorem~5.2]{DPS94}, \cite[Theorem~4.4]{SW04}, \cite[Theorem~2.2, 2.3]{Kane18}}]\label{them:sm:quot}\rm Let $X$ be a smooth projective variety and $\cM$ its unsplit covering family of rational curves. If any $\cM$-curve is free, then $\R_{\geq 0}[\cM]$ is an extremal ray of $\overline{\NE}(X)$ and the associated contraction is a smooth geometric quotient $\varphi: X\to Y$ for $\cM$.
\end{theorem}

\begin{proof} This follows from the same proof as in \cite[Proposition~4.4]{KW23}.
\end{proof}

\begin{theorem}\label{them:unsplit:mor} Let $X$ be a smooth projective variety and $\cM$ its unsplit covering family of rational curves. If $\deg_{(-K_X)}\cM\geq \dim X-2$, then $\R_{\geq 0}[\cM]$ is an extremal ray of $\overline{\NE}(X)$ and the associated contraction is a geometric quotient $\varphi: X\to Y$ for $\cM$.   
\end{theorem}

\begin{proof} This follows from Proposition~\ref{prop:Ion:Wis:2} and \cite[Theorem~2]{BCD07}.
\end{proof}

\begin{remark}\label{rem:cycle}\rm 
Instead of the family of rational curves, we often deal with the family of $1$-cycles, like \cite{BCD07}. Theorem~\ref{them:unsplit:mor} also holds for an irreducible closed subset of ${\rm Chow}(X)$ which parametrizes $1$-cycles. See \cite{BCD07} for details.
\end{remark}

\section{Contractions of fiber type} 
\begin{proposition}\label{prop:charact:P1:product} \rm  Let $X$ be a smooth Fano variety of dimension $n$ and $\iota_X\geq2$. Assume a contraction $\varphi: X \to \P^1$ exists. Then $X$ is a product of $\P^1$ and a variety $Z$. 
\end{proposition}

\begin{proof} By \cite{GHS03}, there exists a minimal section $\tilde{\ell} \subset X$ of $\varphi$, that is, $\tilde{\ell} \subset X$ is a section of $\varphi$ whose anticanonical degree is minimal among such sections. Let us take a family of rational curves $\cM \subset \rat^n{(X)}$ containing $[\tilde{\ell}]$. Then it follows from Lemma~\ref{lem:degeneration:curves} that $\cM$ is unsplit. We first claim that $\dim {\rm Locus}(\cM_x)=1$ for any $x \in {\rm Locus}(\cM)$. To prove this, assume the contrary, that is, $\dim {\rm Locus}(\cM_{x_0})\geq 2$ for some $x_0 \in {\rm Locus}(\cM)$. Then we may find a point $o\in \P^1$ such that $\varphi^{-1}(o)\cap {\rm Locus}(\cM_{x_0}) \neq \emptyset$. Then we have 
$$
\dim \left( \varphi^{-1}(o)\cap {\rm Locus}(\cM_{x_0}) \right) \geq \dim \varphi^{-1}(o)+\dim {\rm Locus}(\cM_{x_0})-n\geq 1.
$$ Thus there exits a curve $C \subset \varphi^{-1}(o)\cap {\rm Locus}(\cM_{x_0})$. Applying \cite[II, Corollary~4.21]{Kb}, we obtain $[C] \in \R_{\geq 0}[\cM]$. This implies that any $\cM$-curve is contracted by $\varphi$; this is a contradiction. 

Let $x$ be any point on ${\rm Locus}(\cM)$. By Proposition~\ref{prop:Ion:Wis:2}, we have  
$$\dim {\rm Locus}(\cM_x) \geq \deg_{(-K_X)}\cM+\codim_X{\rm Locus}(\cM) -1\geq  1+\codim_X{\rm Locus}(\cM).. 
$$
Thus, we see that $\cM$ is an unsplit covering family. Denote by $\psi: X \dashrightarrow Z$ the RCC-fibration with respect to $\cM$. By the same argument as in the above claim that $\dim {\rm Locus}(\cM_x)=1$ for any $x \in {\rm Locus}(\cM)$, we may show that $\dim {\rm ChLocus}_m(\cM)_x=1$ for any $x \in X$. Thus \cite[Proposition~1]{BCD07} tells us that $\psi$ is a morphism. This yields that $\psi$ is the contraction associated to the ray $\R_{\geq 0}[\cM]$. Since $\dim {\rm ChLocus}_m(\cM)_x=1$ for any $x \in X$, we see that $\psi$ is a conic bundle; moreover, any fiber of $\psi$ is numerically equivalent, and any $\cM$-curve is a section of $\varphi$. This concludes that $\psi$ is a $\P^1$-bundle.  By construction, we see that $\varphi\times \psi: X \to \P^1\times Z$ is bijective; thus, the Zariski main theorem concludes that $X \cong \P^1 \times Z$.  
\end{proof}

\if0
\begin{theorem}\label{them:smooth:contraction} \rm Let $X$ be a smooth Fano variety of dimension $n \geq 4$ and $\tau_X\geq n-2$. Let $\varphi: X \to Y$ be a contraction of an extremal ray $R$. If $\varphi$ is of fiber type, then $\varphi$ is one of the following:
 \begin{enumerate}
 \item a smooth morphism;
 \item a $\P^{n-3}$-fibration which is not equidimensional;
 \item an equidimensional $Q^{n-2}$-fibration.
 \end{enumerate} 
 Moreover, any fiber of $\varphi$ has codimension at least $2$.
\end{theorem}
\fi

\begin{proof}[Proof of Theorem~\ref{them:1}] Let $\varphi: X\to Y$ be a contraction of an extremal ray $R$ as in Theorem~\ref{them:1}. If $\dim Y =1$, then it follows from Proposition~\ref{prop:charact:P1:product} that $\varphi$ is a smooth morphism; { since $Y$ is isomorphic to $\P^1$, Proposition~\ref{prop:charact:P1:product} implies that $X$ is isomorphic to a product of $\P^1$ and a variety.} Thus we assume that $\dim Y \geq 2$. General fibers of $\varphi$ are smooth Fano varieties, so there exists a dominating family $\cM\subset \rat^n(X)$ such that any $\cM$-curve is contracted by $\varphi$; replacing if necessary, we may assume that the anticanonical degree of $\cM$ is minimal among such families; then $\cM$ is locally unsplit. We have $\deg_{(-K_X)}\cM \leq n+1$ by Mori's bend and break. If $\deg_{(-K_X)}\cM \geq  n$, then Proposition~\ref{prop:Ion:Wis:2} yields that $\dim {\rm Locus}(\cM_x) \geq n-1$ for a general point $x\in X$. This contradicts our assumption that $\dim Y \geq 2$. Thus, we obtain the following inequality:
$$n-1\geq \deg_{(-K_X)}\cM\geq \ell(R). 
$$
Then we see that $\cM$ is unsplit; in fact, if it is not, by Lemma~\ref{lem:degeneration:curves} there exists a degeneration of $\cM$-curves $Z=\sum_{i=1}^sa_i Z_i$ such that $\sum_{i=1}^sa_i>1$. By the minimality of $\deg_{(-K_X)}\cM$, each $Z_i$ is not free; thus by our assumption that $\tau_X\geq n-2$,  we have $$n-1 \geq \deg_{(-K_X)}\cM=-K_X\cdot Z\geq 2(n-2).$$ This is a contradiction.

Assume that any minimal extremal rational curve of the ray $R$ is free. Then $\cM$ consists of these minimal extremal rational curves, and all $\cM$-curves are free; then Theorem~\ref{them:sm:quot} implies that $\varphi: X\to Y$ is a smooth morphism. So, assume that there exists a minimal extremal rational curve of the ray $R$, which is not free. Then, by our assumption, we have $\ell(R)\geq n-2$, and we obtain the inequality
$$
n-1\geq \deg_{(-K_X)}\cM\geq \ell(R)\geq n-2.
$$
For any $x \in X$, again by Proposition~\ref{prop:Ion:Wis:2}, we have 
$$
\dim {\rm Locus}(\cM_x)\geq \deg_{(-K_X)}\cM-1\geq n-3.
$$
By Theorem~\ref{them:unsplit:mor}, $\R_{\geq 0}[\cM]$ is an extremal ray, and the associated contraction is a geometric quotient $\varphi: X\to Y$ for $\cM$. Let us denote by $B$ the jumping locus for fiber-dimension of $\varphi$:
$$B:=\{x \in X\mid \dim \varphi^{-1}(\varphi(x))> \dim X-\dim Y\}.
$$
By \cite[Proposition~1]{BCD07}, we obtain $\dim B\leq n-2$.
Then we claim the following:
\begin{claim}\label{cl:length} $\deg_{(-K_X)}\cM= \ell(R)$.
\end{claim}
\begin{proof}[Proof of the Claim~\ref{cl:length}] Assume that $\deg_{(-K_X)}\cM> \ell(R)$. Then $\deg_{(-K_X)}\cM=n-1$ and $\ell(R)=n-2$. Moreover, it turns out that, for any point $x\in X$, 
$$
n-2\geq \dim X-\dim Y\geq \dim {\rm Locus}(\cM_x)\geq \deg_{(-K_X)}\cM-1=n-2.
$$
This yields $(\dim Y, \dim {\rm Locus}(\cM_x))=(2, n-2)$.
Since $\dim {\rm Locus}(\cM_x)=n-2=\dim X-\dim Y$, \cite{CMSB} implies that general fibers of $\varphi$ are isomorphic to $\P^{n-2}$. Moreover $\varphi$ is equidimensional, because $\dim B \leq n-2$. It is flat since $\varphi$ is an equidimensional morphism between smooth varieties. Following the idea of the proof of \cite[Proposition~3.1]{HNov13}, we shall prove that $\varphi$ is a projective bundle; then we get a contradiction to our assumption $\deg_{(-K_X)}\cM> \ell(R)$. 

To prove $\varphi$ is a projective bundle, assume the contrary. Then $\varphi$ is not smooth. We denote by $S_{\varphi}$ the locus of points of $Y$ over which $\varphi$ is not smooth and by $R_{\varphi}$ the locus of points of $Y$ over which the fibers of $\varphi$ are reducible. According to the rigidity of the projective space \cite[Main Theorem]{Siu89} (see also \cite[Theorem~1']{HM98}), we remark that $\varphi|_{X\setminus \varphi^{-1}(S_{\varphi})}: X\setminus \varphi^{-1}(S_{\varphi})\to Y\setminus S_{\varphi}$ is a $\P^{n-2}$-bundle. It follows from \cite[Theorem~4]{AR14} that $S_{\varphi}$ is of pure codimension one and $S_{\varphi}=\overline{R_{\varphi}}$. Since $\dim Y=2$, the locus $S_{\varphi}$ is a union of curves.

Let $C_1$ be a minimal extremal rational curve of the ray $R$.
Taking a family of rational curves $\cM'\subset \rat^n(X)$ containing $[C_1]$, the minimality of the anticanonical degree $-K_X\cdot C_1$ and Lemma~\ref{lem:degeneration:curves} imply that $\cM'$ is unsplit. By Proposition~\ref{prop:Ion:Wis:2}, for any $x \in {\rm Locus}(\cM')$, we obtain the inequality
\begin{align*} 
2n-2=\dim X+\deg_{(-K_X)}\cM'\leq \dim {\rm Locus}(\cM')+\dim {\rm Locus}(\cM_x')+1.\tag{1}
\end{align*}
Since any $\cM'$-curve is not free, ${\rm Locus}(\cM')$ does not coincide with the whole space $X$. Moreover, for any $x \in {\rm Locus}(\cM')$, ${\rm Locus}(\cM_x')$ is contained in the fiber $\varphi^{-1}(\varphi(x))$. Hence the right hand side of the inequality $\rm (1)$ is at most 
$$(n-1)+(n-2)+1=2n-2.$$ As a consequence, we see that 
$$
\dim {\rm Locus}(\cM')=n-1\,\,\,\,\,\mbox{and}\,\,\,\,\,\dim {\rm Locus}(\cM_x')=n-2.
$$
Since $\cM'$ is unsplit, Mori's bend and break tells us that the evaluation map from the universal family of $\cM_x'$ onto ${\rm Locus}(\cM_x')$ is generically finite; this yields that $\dim \cM_x'=n-3$ for any $x \in {\rm Locus}(\cM')$. Consider the universal family $p: \cU'\to \cM'$ associated to $\cM'$ and its evaluation morphism $q: \cU'\to {\rm Locus}(\cM')$; by dimension count, we obtain
$$
\dim \cM'= \dim \cU'-1=\dim \cM_x' +\dim {\rm Locus}(\cM')-1=(n-3)+(n-1)-1= 2n-5.
$$
We see that ${\rm Locus}(\cM')$ is contained in $\varphi^{-1}(S_{\varphi})$; since we have 
$$\dim {\rm Locus}(\cM')=\dim \varphi^{-1}(S_{\varphi})=n-1,$$ 
${\rm Locus}(\cM')$ is an irreducible component of $\varphi^{-1}(S_{\varphi})$. Moreover $\varphi({\rm Locus}(\cM'))$ coincides with an irreducible component of $S_{\varphi}= \overline{R_{\varphi}}$; thus we may find a point $x_0\in {\rm Locus}(\cM')$ such that $\varphi(x_0)$ is in $R_{\varphi}$. Set $F_0:=\varphi^{-1}(\varphi(x_0))$. Denoting by $D_1, \ldots, D_m~(m>1)$ irreducible components of the fiber $(F_0)_{\rm red}$ with reduced structure, there exist positive integers $a_i$'s such that the cycle $[F_0]$ is algebraically equivalent to $\sum_{i=1}^ma_i[D_i]$ with $\sum_{i=1}^ma_i>1$. Set $D_1^o:=\left(D_1\setminus \bigcup_{i\neq 1}D_i   \right)$, which is a nonempty open subset of $D_1$. Without loss of generality, we may assume that $D_1 \subset {\rm Locus}(\cM')$. 
For any point $x\in D_1^o$, any $\cM_x'$-curve is contained in $F_0$ and passes through the point $x \in X$; this implies that ${\rm Locus}(\cM_x')$ is contained in $D_1$; since ${\rm Locus}(\cM_x')$ and $D_1$ have the same dimension and $D_1$ is irreducible, we see that ${\rm Locus}(\cM_x')=D_1$. By a dimension count, we obtain
$$
\dim p(q^{-1}(D_1^o))= \dim q^{-1}(D_1^o)-1=\dim \cM_x' +\dim D_1^o-1= 2n-6.
$$
This implies that there exists a family of rational curves $\cN\subset \rat^n(D_1)$ such that $\dim \cN$ is at least $2n-6$ and $\deg_{(-K_X|_{D_1})}\cN=\ell(R)$. Denoting by $\nu: \tilde{D_1}\to D_1$ the normalization of $D_1$, our family $\cN$ lifts to $\tilde{D_1}$. We denote by $\tilde{\cN}$ the lifting of $\cN$ to $\tilde{D_1}$. Then by applying \cite[Theorem~2.1]{HNov13}, we see that $D_1$ is isomorphic to $\P^{n-2}$ and $\tilde{\cN}$ is the family of lines on $\P^{n-2}$.

Let $F$ be a general fiber of $\varphi$. By \cite[I,~Proposition~3.12]{Kb}, we obtain 
\begin{eqnarray*}
(n-1)^{n-2} &=& (-K_F)^{n-2}=[F]\cdot (-K_X)^{n-2}=[F_0]\cdot(-K_X)^{n-2} \\
 &=& \sum_{i=1}^m a_i\left([D_i]\cdot (-K_X)^{n-2}\right)>[D_1]\cdot (-K_X)^{n-2}=\left(\nu^{\ast}(-K_X)\right)^{n-2}.
\end{eqnarray*}
For a positive integer $b$ such that $\nu^{\ast}(-K_X)=\cO_{\P^{n-2}}(b)$, the above inequality tells us that $b^{n-2}<(n-1)^{n-2}$; thus we see that $b<n-1$. On the other hand, for a line $f \subset \P^{n-2}$,  we have 
$$
n-2=\ell(R)\leq (-K_X)\cdot \nu_{\ast}(f) =\nu^{\ast}(-K_X)\cdot f =b<n-1.
$$
Hence, we obtain $b=n-2$. 

Recall that $\cM$ is an unsplit covering family; for any point $x \in D_1^o$ there exists an $\cM_x$-curve $C$; we see that $C$ is contained in $D_1$; it lifts to a curve $\tilde{C}$ in $\tilde{D_1}\cong \P^{n-2}$. Then we obtain the following:
$$
(n-2)\cO_{\P^{n-2}}(1)\cdot \tilde{C}=\nu^{\ast}(-K_X)\cdot \tilde{C}=(-K_X)\cdot {C}=n-1.
$$
This is a contradiction, because $\cO_{\P^{n-2}}(1)\cdot \tilde{C}$ is an integer and $n\geq 4$.
\end{proof}

Let $F$ be a smooth general fiber of $\varphi$. As we proved, we have the following inequalities:
\begin{eqnarray*}
n-2 \geq \dim F \geq n-3, \\
\quad n-1\geq \deg_{(-K_X)}\cM=\ell(R)\geq n-2.
\end{eqnarray*}
By \cite{CMSB} and \cite{DH17} (see also \cite{Mi}), $F$ is isomorphic to $\P^{n-2}, Q^{n-2}$ or $\P^{n-3}$. Assume that $F$ is isomorphic to $\P^{n-2}$ or $Q^{n-2}$; then $\varphi$ is equidimensional, because $\dim B \leq n-2$. Thus if $F$ is isomorphic to $\P^{n-2}$, \cite{HNov13} yields that $\varphi$ is a $\P^{n-2}$-bundle; this contradicts to our assumption that $\varphi$ is not smooth. Consequently, we see that $\varphi$ is an equidimensional $Q^{n-2}$-fibration. Finally, we assume that $F$ is isomorphic to $\P^{n-3}$. If $\varphi$ is equidimensional, then again by \cite{HNov13} $\varphi$ should be a $\P^{n-3}$-bundle; it is a contradiction because  $\varphi$ is not smooth. As a consequence, 
$\varphi$ is a $\P^{n-3}$-fibration which is not equidimensional. 
\end{proof} %equidimensional に関しては\cite[Lemma~3.9, 3.10]{casa08}も参照．

\if0
By the same argument as in Theorem~\ref{them:smooth:contraction}, we obtain the following:

 \begin{theorem}\label{them:smooth:contraction:4fold} \rm Let $X$ be a smooth Fano fourfold of $\iota_X\geq 2$ and $\rho_X \geq 2$. Let $\varphi: X \to Y$ be an elementary contraction of fiber type, then $\varphi$ is one of the following:
 \begin{enumerate}
 \item a smooth morphism;
 \item an inequidimensional $\P^1$-fibration;
 \item an equidimensional $Q^2$-fibration.
 \end{enumerate} 
 Moreover, any fiber of $\varphi$ has codimension at least $2$.
\end{theorem}
\fi

\section{Examples} Let $X$ be a variety as in Theorem~\ref{them:2} or Theorem~\ref{them:3}. This section aims to compute invariants $\tau_X, \iota_X, \ell_X, i_X$ and $\nu_X$ of such $X$.

%\subsection{Fano varieties with large $\tau_X$ and a birational contraction }

\begin{theorem}\label{them:bir:contr:n-1} \rm  Let $X$ be a smooth Fano variety of dimension $n\geq 4$ and assume that $X$ admits a birational elementary contraction. Then the following are equivalent to each other:
\begin{enumerate} 
\item $X$ is isomorphic to $\P(\cO_{\P^{n-1}}\oplus \cO_{\P^{n-1}}(1))$;
\item $\nu_X=2$;
\item $\tau_X\geq n-1$;
\item $\tau_X=n-1$.
\end{enumerate}
\end{theorem}

\begin{proof} Remark that $\P(\cO_{\P^{n-1}}\oplus \cO_{\P^{n-1}}(1))$ is the blow up of $\P^n$ at a point. In the following, we denote by $E$ its exceptional divisor. 

$\rm (i)\Rightarrow (ii)$ Assume that $X$ is isomorphic to $\P(\cO_{\P^{n-1}}\oplus \cO_{\P^{n-1}}(1))$. We denote by $\cO_X(1)$ the tautological line bundle of $\P(\cO_{\P^{n-1}}\oplus \cO_{\P^{n-1}}(1))$ and by $\pi: X\to \P^{n-1}$ the natural projection. Then we have the following exact sequence:
$$0\to \cO_X(2)\otimes \pi^{\ast}(T_{\P^{n-1} }(-1))\to \wedge^2 T_X \to \wedge^2\pi^{\ast}T_{\P^{n-1}}\to 0.$$
Since the bundles $\cO_X(2), \pi^{\ast}(T_{\P^{n-1} }(-1))$ and $\pi^{\ast}T_{\P^{n-1}}$ are nef, so is $\wedge^2 T_X$. Moreover, Theorem~\ref{them:sm:quot} tells us that $T_X$ is not nef.

$\rm (ii)\Rightarrow (iii)$ This follows from Lemma~\ref{lem:nonfree}. 

$\rm (iii)\Rightarrow (i)$, $\rm (iv)\Rightarrow (i)$  These follow from the same argument as in \cite{Yas14} or \cite[Theorem~1.7]{Wat21}.

$\rm (i)\Rightarrow (iv)$ If $X$ is $\P(\cO_{\P^{n-1}}\oplus \cO_{\P^{n-1}}(1))$, the above arguments imply that $\tau_X\geq n-1$. For a line $\ell$ in the exceptional divisor $E\cong\P^{n-1}$ of $X$, we have $-K_X\cdot \ell=n-1$. Since $\ell$ is not free, our assertion holds.
\end{proof}

\begin{lemma}\label{eg:k:tau}\rm Let $X$ be a variety other than (i) among the varieties in Theorem~\ref{them:2}. If $\wedge^3T_X$ is nef, then $\nu_X=3$ and $\tau_X=n-2$.
\end{lemma}

\begin{proof} Combining $\nu_X=3$ with Lemma~\ref{lem:nonfree} and Theorem~\ref{them:bir:contr:n-1}, we see that $\tau_X=n-2$.
\end{proof}

\begin{example}\label{eg:bir:1} Let $X=\P(\cO_{\P^{n-1}}\oplus \cO_{\P^{n-1}}(1))$ for $n\geq 4$. We denote by $\cO_X(1)$ the tautological line bundle of $X$ and by $\pi: X\to \P^{n-1}$ the natural projection. Then $X$ is the blow-up of $\P^n$ at a point. By Proposition~\ref{prop:div:iota}, $\iota_X=\ell_X=2$. We have $\cO_X(-K_X)= \cO_X(2)\otimes\pi^{\ast}\cO_{\P^{n-1}}(n-1)$. Since ${\rm Pic}(X)\cong \Z[\cO_X(\xi)]\oplus \Z[\pi^{\ast}\cO_{\P^{n-1}}(1)]$, we have
\begin{equation}
  i_X=
  \begin{cases}
    1 & \text{if} ~n ~\text{is~ even}, \\ \nonumber
    2 & \text{if} ~n ~\text{is ~odd}.
  \end{cases}
\end{equation}
By Theorem~\ref{them:bir:contr:n-1}, we see that $\nu_X=2$ and $\tau_X= n-1$.
\end{example}

\begin{example}\label{eg:Q:blup} Let $X=\P(\cO_{Q^{n-1}}\oplus \cO_{Q^{n-1}}(1))$ for $n\geq 4$. We denote by $\cO_X(1)$ the tautological line bundle of $X=\P(\cO_{Q^{n-1}}\oplus \cO_{Q^{n-1}}(1))$ and by $\pi: X\to Q^{n-1}$ the natural projection. 
We have the following exact sequence
\begin{equation}
0\to T_{\pi}\to T_X \to \pi^{\ast}T_{Q^{n-1}} \to 0.  \nonumber
\end{equation}
By \cite[II~ Exercises~5.16]{Har}, the above sequence implies the following exact sequence
\begin{equation}\label{eq:3}
0\to T_{\pi}\otimes\wedge^2\pi^{\ast}T_{Q^{n-1}}\to \wedge^3 T_X \to \wedge^3\pi^{\ast}T_{Q^{n-1}} \to 0.
\end{equation}
On the other hand, we have 
$$T_{\pi}\otimes\wedge^2\pi^{\ast}T_{Q^{n-1}}\cong \cO_X(2)\otimes \pi^{\ast}\left\{\left(\wedge^2 T_{Q^{n-1}}\right)\otimes \cO_{Q^{n-1}}(-1)  \right\}.$$
From the standard exact sequence $0\to T_{Q^{n-1}}\to T_{\P^n}|_{Q^{n-1}}\to \cO_{Q^{n-1}}(2)\to 0$ and \cite[II~ Exercises~5.16]{Har}, we obtain a surjection 
$$
\wedge^3T_{\P^{n}}|_{Q^{n-1}} \to \left(\wedge^2 T_{Q^{n-1}} \right) \otimes \cO_{Q^{n-1}}(2).
$$
Tensoring $\cO_{Q^{n-1}}(-3)$ with the above surjection, we see that $\left(\wedge^2 T_{Q^{n-1}} \right) \otimes \cO_{Q^{n-1}}(-1)$ is nef. Then (\ref{eq:3}) implies that $\wedge^3 T_X$ is nef. By Lemma~\ref{eg:k:tau}, $\nu_X=3$ and $\tau_X=n-2$.

Since $X$ is the blow up of the cone over $Q^{n-1}$ at the vertex, Proposition~\ref{prop:div:iota} implies $\iota_X=\ell_X=2$. Moreover we have $$\cO_X(-K_X)\cong \cO_X(2)\otimes\pi^{\ast}\cO_{Q^{n-1}}(n-2).$$ Since ${\rm Pic}(X)\cong \Z[\cO_X(1)]\oplus \Z[\pi^{\ast}\cO_{Q^{n-1}}(1)]$, we have
\begin{equation}
  i_X=
  \begin{cases}
    2 & \text{if} ~n ~\text{is~ even}, \\ \nonumber
    1 & \text{if} ~n ~\text{is ~odd}.
  \end{cases}
\end{equation}

\end{example}

\begin{example}\label{eg:P:blup} Let $X=\P(\cO_{\P^{n-1}}\oplus \cO_{\P^{n-1}}(2))$ for $n\geq 4$. Then by the same argument as in Example~\ref{eg:Q:blup} we see that $(\nu_X, \tau_X, \iota_X, \ell_X)=(3, n-2, 2, 2)$ and 
\begin{equation}
  i_X=
  \begin{cases}
    2 & \text{if} ~n ~\text{is~ even}, \\ \nonumber
    1 & \text{if} ~n ~\text{is ~odd}.
  \end{cases}
\end{equation}
\end{example}

\begin{example}\label{eg:P:blup:l} Let $X$ be the blow up of $\P^n$ along a line $\ell$, that is, $X=\P(\cO_{\P^{n-2}}^{\oplus 2}\oplus \cO_{\P^{n-2}}(1))$ for $n\geq 4$. We denote by $\cO_X(1)$ the tautological line bundle of $X=\P(\cO_{\P^{n-2}}^{\oplus 2}\oplus \cO_{\P^{n-2}}(1))$ and by $\pi: X\to \P^{n-2}$ the natural projection. We have the following exact sequence
\begin{equation}
0\to T_{\pi}\to T_X \to \pi^{\ast}T_{\P^{n-2}} \to 0.  \nonumber
\end{equation}By \cite[II~ Exercises~5.16]{Har}, there exists a vector bundle $F$ on $X$ fitting in the following exact sequences
\begin{equation}\label{eq:4}
0\to F\to \wedge^3 T_X \to \wedge^3\pi^{\ast}T_{\P^{n-2}} \to 0,
\end{equation}
\begin{equation}\label{eq:5}
0\to \cO_X(3)\otimes\pi^{\ast}\left(T_{\P^{n-2}}\right)\to F \to T_{\pi}\otimes \wedge^2\pi^{\ast}T_{\P^{n-2}} \to 0.
\end{equation}
On the other hand, tensoring the relative Euler sequence with $\wedge^2\pi^{\ast}T_{\P^{n-2}}$, we obtain a surjection 
$$
\cO_X(1)\otimes \pi^{\ast}\left(\left(\wedge^2T_{\P^{n-2}}\right)^{\oplus 2}\oplus \left(\wedge^2T_{\P^{n-2}}\right)(-1)\right)\to T_{\pi}\otimes \wedge^2 \pi^{\ast}T_{\P^{n-2}}. $$
 This implies that $T_{\pi}\otimes \wedge^2 \pi^{\ast}T_{\P^{n-2}}$ is nef. Since $\cO_X(3)\otimes\pi^{\ast}\left(T_{\P^{n-2}}(-1)\right)$ is also nef, (\ref{eq:5}) yields that $\wedge^3 T_X$ is nef. By Lemma~\ref{eg:k:tau}, $\nu_X=3$ and $\tau_X=n-2$.

From now on, we follow the computation of \cite[Proof of Theorem~1]{Tsu12}. Remark that $X$ is the blow-up of $\P^n$ along a line $\ell$, denoted by $\varphi: X\to \P^n$. Since $N_{\ell/\P^n}\cong \cO_{\P^1}(1)^{\oplus n-1}$, the exceptional divisor $E$ is isomorphic to $\P(N_{\ell/\P^n}^{\vee})\cong \P^1 \times \P^{n-2}$; moreover the restrictions $\pi|_E: E\to \P^{n-2}$ and $\varphi|_E: E\to \P^1$ give the two natural projections. Denoting by $e$ a line in a fiber $\P^{n-2}$ of $\varphi|_{E}: E\to \ell$ and by $f$ a line contained in a fiber $\P^{2}$ of $\pi$, the Kleiman-Mori cone $\overline{ NE}(X)$ is generated by these two lines:
$$
\overline{ NE}(X)=\R_{\geq 0}[e]+\R_{\geq 0}[f].
$$
Since $$\cO_X(-K_X)\cong\cO_X(3)\otimes\pi^{\ast}\cO_{\P^{n-2}}(n-2),$$ we have $-K_X\cdot f=3$ and $-K_X\cdot \ell=n-2$. Moreover we see that $i_X=3$ if $n\equiv 2~(3)$; otherwise $i_X=1$. 
For any projective curve $\Gamma \subset X$, we claim that $-K_X\cdot \Gamma\geq 2$ if $n=4$, and otherwise $-K_X\cdot \Gamma\geq 3$. To prove this, take real numbers $a, b$ such that $\Gamma=ae+bf$ in $N_1(X)$. For a hyperplane $H$ of $\P^{n-2}$, $a=\Gamma\cdot \pi^{\ast}H\geq 0$ and $b=\Gamma\cdot \xi\geq 0$; this yields that $a$ and $b$ are nonnegative integers. Assume $\Gamma$ is not contained in $E$. For a hyperplane $H'$ of $\P^n$, we remark that $\varphi^{\ast}H'\cdot \Gamma\geq E\cdot \Gamma$. Since $-K_X=\varphi^{\ast}(-K_{\P^n})-(n-2)E$, we have
$$
-K_X\cdot \Gamma=\left\{(n+1)\varphi^{\ast}H'-(n-2)E\right\}\cdot \Gamma\geq 3\varphi^{\ast}H'\cdot \Gamma\geq 3.
$$
Assume $\Gamma$ is contained in $E\cong \P^1 \times \P^{n-2}$. Then we have 
$$
-K_X\cdot \Gamma=(-K_E+E|_E)\cdot \Gamma=\{(n-1)a+2b\}+ \{-a+b\}=(n-2)a+3b.
$$
This yields that $-K_X\cdot \Gamma\geq 2$ if $n=4$; otherwise $-K_X\cdot \Gamma\geq 3$.
As a consequence, we have 
\begin{equation}
  \iota_X=\ell_X=
  \begin{cases}
    2 & \text{if} ~n=4, \\ \nonumber
    3 & \text{if} ~n\geq 5.
  \end{cases}
\end{equation}
\end{example}

\begin{example}\label{eg:Q:blup:l} Let $X$ be the blow up of $Q^n\subset \P^{n+1}$ along a line $\ell$ for $n\geq 4$. Consider the projection of $Q^n$ from $\ell$; we denote it by $q: Q^n \dashrightarrow \P^{n-1}$. We can eliminate the points of indeterminacy of $q$ by taking the blow-up of $Q^n$ along $\ell$; then we have the following commutative diagram:

\[
  \xymatrix{
   X \ar[r]^{\varphi} \ar[d]_{\pi}  & Q^n \ar@{-->}	[dl]^{q}	 \\
   \P^{n-1}  &   }
\]

For any $p\in \P^{n-1}$, there exists a plane $\Lambda\subset \P^{n+1}$ containing $\ell$ such that the fiber $q^{-1}(p)=\Lambda\cap Q^n\setminus \ell$. It turns out that $\pi$ is a $\P^1$-fibration which is not equidimensional. By Theorem~\ref{them:1}, $\wedge^3T_X$ is not nef if $n> 4$. Thus, we assume that $n=4$. In this case, we can prove that $\wedge^3T_X$ is nef, but we omit the proof because it is routine work (see \cite[Example~1.6]{Yas14}). By Lemma~\ref{eg:k:tau}, $\nu_X=3$ and $\tau_X=2$. On the other hand, we have $$-K_X=\varphi^{\ast}(-K_{Q^4})-2E=2\left(2\varphi^{\ast}H_{Q^4}-E\right),$$ where $E$ is the exceptional divisor of $\varphi$ and $H_{Q^4}$ is a hyperplane section of $Q^4$. This implies that $i_X=2$. Applying Lemma~\ref{lem:pi:4fold} below, we see that $\iota_X=2$. 
\end{example}

\begin{lemma}\label{lem:pi:4fold}\rm  If $X$ is a smooth Fano $4$-fold with $i_X=2$, then $\iota_X=2$. 
\end{lemma}

\begin{proof} By definition, $\iota_X$ is a multiple of $i_X$; then our assertion follows from Theorem~\ref{them:fano:collect}.
\end{proof}

\begin{example}\label{eg:Q4:blup} Let $X$ be the blow up $\varphi: X\to Q^4$ of a smooth quadric $4$-fold $Q^4$ along a smooth projective curve $C$. Assume that $X$ is a Fano variety. Then we have $-K_X=4\varphi^{\ast}H+2E$, where $H$ is a hyperplane section of $Q^4$ and $E$ is the exceptional divisor of $\varphi$. Thus we see that $i_X=2$. By Lemma~\ref{lem:pi:4fold}, we have $\iota_X=2$. By \cite[Theorem~1]{Tsu12}, we also see $\ell_X=2$. 
\end{example}

\begin{example}\label{eg:Q:blup:conic} Let $X$ be the blow-up of $Q^n\subset \P^{n+1}$ along a conic which is not on a plane contained in $Q^n$ for $n\geq 4$. By the same argument as in  Example~\ref{eg:Q:blup:l}, we see that $\wedge^3T_X$ is not nef if $n>4$. If $n=4$, we see that $\wedge^3T_X$ is nef (see \cite[Example~1.9]{Yas14}) and $(\nu_X, \tau_X, i_X, \iota_X, \ell_X)=(3, 2, 2, 2, 2)$. 
\end{example}

\begin{example}\label{eg:PtimesBl:blup} For $n\geq 4$, let $Y$ be the blow up of $\P^{n-1}$ at a point and $X$ be the product of $\P^1$ and $Y$. We denote by $p_1: X \to \P^1$ and $p_2 : X \to \P( \cO_{\P^{n-2}}(1)\oplus \cO_{\P^{n-2}})$ the first and second projections, respectively. Then we see that $\wedge^3 T_X\cong \left(p_1^{\ast}T_{\P^1}\otimes \wedge^2  {p_2}^{\ast}T_{Y}\right)\oplus \wedge^3  {p_2}^{\ast}T_{Y}$ is nef.

By Lemma~\ref{lem:cone} below and Example~\ref{eg:bir:1}, we see that $\iota_X=\ell_X=1$ and 
\begin{equation}
  i_X=
  \begin{cases}
    2 & \text{if} ~n ~\text{is~ even}, \\ \nonumber
    1 & \text{if} ~n ~\text{is ~odd}.
  \end{cases}
\end{equation} 

\end{example}

%%%%%%%%%%%%%%%%%%%%%%%%%%%%%%%%%%%%%%%%%%%%%%
%%%%%%%%%%%%%%%%%%%%%%%%%%%%%%%%%%%%%%%%%%%%%%
%%%%%%%%%%%%%%%%%%%%%%%%%%%%%%%%%%%%%%%%%%%%%%

\begin{lemma}\label{lem:cone}\rm Let $X$ be a product of $\P^m~(m>0)$ and a smooth Fano variety $Y$. Let $p: X\to \P^m$ and $q: X\to Y$ be natural projections. Assume that the Kleiman-Mori cone $\overline{ NE}(Y)$ is generated by rational curves $\ell_1, \ldots, \ell_r$. Then there exist rational curves $\tilde{\ell}, \tilde{\ell_1}, \ldots, \tilde{\ell_r}$ such that 
\begin{enumerate}
\item $\tilde{\ell}$ is contracted by $q$; 
\item $q_{\ast}(\tilde{\ell_i})=\ell_i$ for any $i$;
\item $\overline{ NE}(X)$ is generated by rational curves $\tilde{\ell}, \tilde{\ell_1}, \ldots, \tilde{\ell_r}$.
\end{enumerate}
\end{lemma}

\begin{proof} Let $\Gamma$ be a projective curve on $X$. Then there exist nonnegative real numbers $c_1,\ldots, c_r$ such that $q_{\ast}=\sum_{i=1}^rc_i\ell_i$ in $\overline{ NE}(Y)$. Since $\Gamma-\sum_{i=1}^rc_i\tilde{\ell_i}$ is contained in ${\rm Ker}(q_{\ast})=\langle \tilde{\ell} \rangle_\R \subset \overline{ NE}(X)$, there exists $a\in \R$ such that $\Gamma=a\tilde{\ell}+\sum_{i=1}^rc_i\ell_i$. Then we have $a=p^{\ast}(\cO_{\P^m}(1))\cdot \Gamma\geq 0$.
\end{proof}

%Let $X$ be a smooth Fano variety of dimension $n\geq 4$. Assume that $\tau_X\geq n-2$ and $\rho_X>1$. In the following table, we list various invariants of $X$. 

Let $X$ be a smooth Fano variety of dimension $n\geq 4$. Assume that $\tau_X\geq n-2$ and $\rho_X>1$. Putting all the above together, we get the following Table~$1$.
In Table~$1$, ${Bl}_{\ell}(\P^n)$ is the blow up of $\P^n$ along a line $\ell$; ${Bl}_{\ell}(\Q^4)$ is the blow up of $Q^4$ along a line $\ell$; ${Bl}_{C}(\Q^4)$ is the blow up of $Q^4$ along a conic $C$ which is not on a plane contained in $Q^4$.
%\newpage
\begin{table}[H]
  \caption{Fano varieties with large $\tau_X$ and a birational contraction}
  \label{table:data_type}
  \centering
  \begin{tabular}{cccccc}
    \hline
    $X$  & $\tau_X$  &  $\iota_X=\ell_X$ & $i_X$ & $\nu_X$  \\
    \hline \hline
    $\P(\cO_{\P^{n-1}}\oplus \cO_{\P^{n-1}}(1))$                  & $n-1$  & $2$ & \begin{tabular}{c}
$1$ if $n$ is even.\\  \hline
$2$ if $n$ is odd.
\end{tabular} & $2$   \\ \hline
    $\P(\cO_{Q^{n-1}}\oplus \cO_{Q^{n-1}}(1))$                    & $n-2$ & $2$ & \begin{tabular}{c}
$2$ if $n$ is even.\\  \hline
$1$ if $n$ is odd.
\end{tabular} & $3$ \\ \hline
    $\P(\cO_{\P^{n-1}}\oplus \cO_{\P^{n-1}}(2))$                  & $n-2$  & $2$ & \begin{tabular}{c}
$2$ if $n$ is even.\\  \hline
$1$ if $n$ is odd. \end{tabular} & $3$ \\ \hline
    ${Bl}_{\ell}(\P^n)=\P(\cO_{\P^{n-2}}^{\oplus 2}\oplus \cO_{\P^{n-2}}(1))$  & $n-2$  & \begin{tabular}{c}
$2$ if $n=4$.\\  \hline
$3$ if $n\geq 5$.
\end{tabular} & \begin{tabular}{c}
$3$ if $n\equiv 2~(3)$.\\  \hline
$1$ otherwise.
\end{tabular} & $3$  \\ \hline
${Bl}_{\ell}(\Q^4)$                  & $n-2=2$  & $2$ & $2$ & $3$   \\ \hline
${Bl}_{C}(\Q^4)$                  & $n-2=2$  & $2$ & $2$ & $3$   \\ \hline
    $\P^1\times \P(\cO_{\P^{n-2}}\oplus \cO_{\P^{n-2}}(1))$ & $n-2$  & $2$  & \begin{tabular}{c}
$2$ if $n$ is even.\\  \hline
$1$ if $n$ is odd.
\end{tabular} & $3$ \\ \hline
  \end{tabular}
\end{table}

Let $X$ be a variety as in Theorem~\ref{them:3} and assume $n=\dim X \geq 5$. If $X$ is not isomorphic to $\P(\cO_{\P^{3}}^{\oplus 2}\oplus \cO_{\P^{3}}(1))$, then $X$ is homogeneous; thus the tangent bundle $T_X$ is nef. One can obtain the following table: 

\begin{table}[H]
  \caption{Fano varieties with large $\iota_X$ and $\rho_X>1$}
  \label{table:data_type}
  \centering
  \begin{tabular}{cccccc}
    \hline
    $X$  & $\tau_X$  &  $\iota_X=\ell_X$ & $i_X $ & $\nu_X$   \\
    \hline \hline
        $\P^3\times \P^3$                  & $+\infty$  & $4$ & $4$ & $1$  \\ \hline
    $\P(\cO_{\P^{3}}^{\oplus 2}\oplus \cO_{\P^{3}}(1))$  & $3$  & $3$ & $1$ & $3$  \\ \hline
    $\P^3\times \P^2$                  & $+\infty$ & $3$ & $1$ & $1$  \\ \hline
    $Q^3\times \P^2$                  & $+\infty$  & $3$ & $3$ & $1$  \\ \hline
    $\P(T_{\P^3})$  & $+\infty$  & $3$ & $3$ & $1$  \\ \hline
  \end{tabular}
\end{table}

\begin{example}\label{eg:(1,1)} Let $X$ be a divisor on $\P^2 \times \P^3$ of bidegree $(1,1)$. Denote by the homogeneous coordinates of $\P^2$ and $\P^3$ by $(X_0, X_1, X_2)$ and $(Y_0, Y_1, Y_2, Y_3)$ respectively. 
Changing the coordinates if necessary, we may assume that 
$$X=\{X_0Y_0+X_1Y_1+X_2Y_2=0\}\subset \P^2 \times \P^3.$$  
Let $p_1: X\to \P^2$ and $p_2: X\to \P^3$ be the two natural projections. We denote by $H_1$ a hyperplane of $\P^2$ and by $H_2$ a hyperplane of $\P^3$. We see that $p_1$ is a $\P^2$-bundle. On the other hand, a fiber of $p_2$ at $(0, 0, 0, 1)\in \P^3$ is isomorphic to $\P^2$, while all other fibers of $p_2$ are lines in $\P^2$. Let $\ell_1$ be a line in a fiber of $p_1$ and $\ell_2$ be a general fiber of $p_2$. Then we have 
$$
\overline{ NE}(X)=\R_{\geq 0}[\ell_1]+\R_{\geq 0}[\ell_2], 
$$  
$$
-K_X\cdot \ell_1=3,\quad-K_X\cdot \ell_2=2,\quad p_1^{\ast}\cO_{\P^2}(1)\cdot \ell_2=p_2^{\ast}\cO_{\P^3}(1)\cdot \ell_1=1.
$$
For any projective curve $\Gamma\subset X$, there exist nonnegative integers $a, b$ such that $\Gamma=a\ell_1+b\ell_2 \in \overline{ NE}(X)$. Thus we have $-K_X\cdot \Gamma =3a+2b\geq 2$. This implies that $\iota_X=\ell_X=2$. Since $-K_X=p_1^{\ast}H_1+ p_2^{\ast}H_2$, $X$ is a Fano variety with $i_X=1$.    
\end{example}

Let $X$ be a smooth Fano $4$-fold as in Theorem~\ref{them:3}. 
Using the above results so far, we obtain the following table. In this table, if X is a Fano fourfold of $i_X=2$, the values of $\tau_X$ and $\nu_X$ differ depending on which variety $X$ is isomorphic to, so they are abbreviated as $\ast$.
\begin{table}[H]
  \caption{Fano $4$-folds with $\iota_X\geq2$ and $\rho_X>1$}
  \label{table:data_type}
  \centering
  \begin{tabular}{cccccc}
    \hline
    $X$  & $\tau_X$  &  $\iota_X=\ell_X$ & $i_X $ & $\nu_X$   \\
    \hline \hline
     $\P(\cO_{\P^{3}}\oplus \cO_{\P^{3}}(1))$  & $3$  & $2$ & $1$ & $2$  \\ \hline
    $\P(\cO_{\P^{2}}^{\oplus 2}\oplus \cO_{\P^{2}}(1))$  & $2$  & $2$ & $1$ & $3$  \\ \hline
     a divisor on $\P^2\times \P^3$ of bidegree $(1, 1)$                  & $2$  & $2$ & $1$ & $3$  \\ \hline
     $\P^2\times \P^2$                & $+\infty$  & $3$ & $3$ & $1$  \\ \hline

    $\P^1\times Q^3$             & $+\infty$ & $2$ & $1$ & $1$  \\ \hline
    a Fano fourfold of $i_X=2$             & $\ast$  & $2$ & $2$ & $\ast$  \\ \hline
  \end{tabular}
\end{table}

\section{Contractions of birational type}

\if0
%The main result of this section is the following:
\begin{theorem}[\cite{Yas14}]\label{them:bir:contr} \rm  Let $X$ be a smooth Fano variety of dimension $n=\dim X\geq 4$ and $\tau_X=n-2$. If $X$ admits a birational elementary contraction, then $X$ is one of the following:
\begin{enumerate}
%\item $\P(\cO_{\P^{n-1}}\oplus \cO_{\P^{n-1}}(-1))$; %In this case, $\wedge^2T_X$ is nef;
\item $\P(\cO_{Q^{n-1}}\oplus \cO_{Q^{n-1}}(-1))$;
\item $\P(\cO_{\P^{n-1}}\oplus \cO_{\P^{n-1}}(-2))$;
\item the blow-up of $\P^n$ along a line;
\item the blow-up of $Q^4$ along a line;
\item the blow-up of $Q^4$ along a conic which is not on a plane contained in $Q^4$;
\item $\P^1\times \P(\cO_{\P^{n-2}}\oplus\cO_{\P^{n-2}}(1))$.
\end{enumerate}
\end{theorem}
\fi
In this section, we prove Theorem~\ref{them:2}. Yasutake proved it in \cite{Yas14} when $\wedge^3 T_X$ is nef; however, it seems that his proof contains a nontrivial gap (see Remark~\ref{rem:Yasutake}); thus, following the idea of the proof of \cite{Yas14}, we fill the gap. 
In the rest of this section, we always assume the following: 
 \begin{assumption}\label{ass:bir} \rm Let $X$ be a smooth Fano variety of dimension $n \geq 4$ and $\tau_X\geq n-2$. Moreover, assume that $X$ admits a birational contraction $\varphi: X \to Y$ of an extremal ray $R$. We denote by $E$ the exceptional locus of $\varphi$ and by $F$ a fiber of $\varphi|_{E}: E\to W:=\varphi(E)$. 
\end{assumption}

\begin{lemma}\label{lem:bir:type} \rm  $(\dim E, \dim W)=(n-1, 0)$ or $(n-1,1)$. We have $\ell(R)=n-2$ in the latter case. 
\end{lemma}
 
\begin{proof} Since any rational curve $C$ contracted by $\varphi$ is not free, we have $-K_X\cdot C \geq n-2$. Then Proposition~\ref{prop:Ion:Wis} yields that 
$$
2(n-1)\geq \dim E +\dim F \geq n+\ell(R)-1\geq n+(n-2)-1=2n-3.
$$ 
As a consequence, we obtain our assertion.
\end{proof}

\begin{lemma}\label{lem:E-negative:ray} \rm There exists an extremal ray $R'$ such that $E\cdot R'>0$. Moreover, every fiber of the contraction of the ray $R'$ has dimension at most two.
\end{lemma}
 
\begin{proof} Take a rational curve $C_0 \subset X$ such that $E\cdot C_0>0$. According to the Cone theorem, there exist finitely many extremal rational curves $C_i$'s and positive real numbers $a_i$'s such that $[C_0]$ is numerically equivalent to $\sum a_i [C_i]$; since we have 
$$
\sum a_i (E\cdot C_i)=E\cdot C_0 >0,
$$
there exists an extremal rational curve $C_{i_0}$ such that $E\cdot C_{i_0}>0$ as desired. 

The latter part of our assertion follows from standard dimension count.
\end{proof}

In the following, we denote by $\psi: X\to Y'$ the contraction associated to $R'$. 

\begin{proposition}\label{prop:bir:fiber:n-1:0} \rm  Assume that $\psi$ is of fiber type. If $(\dim E, \dim W)=(n-1, 0)$, then $X$ is isomorphic to one of the following:
\begin{enumerate}
\item $\P(\cO_{\P^{n-1}}\oplus \cO_{\P^{n-1}}(1))$. %In this case, $\wedge^2T_X$ is nef.
\item $\P(\cO_{Q^{n-1}}\oplus \cO_{Q^{n-1}}(1))$.
\item $\P(\cO_{\P^{n-1}}\oplus \cO_{\P^{n-1}}(2))$.
\end{enumerate}
\end{proposition}
 
\begin{proof} We claim that $\psi(E)$ coincides with $Y'$. If not, a curve $D$ exists in the exceptional divisor $E$, which is contracted by $\psi$; this is a contradiction because the curve $[D]$ lives in the rays $R$ and $R'$. Then it turns out that $\psi^{-1}(y')\cap E$ is a non-empty finite set for any point $y' \in Y'$. This yields an inequality
$$
0=\dim \psi^{-1}(y')\cap E\geq \dim \psi^{-1}(y')+\dim E -n=\dim \psi^{-1}(y')-1.
$$ 
This tells us that $\dim \psi^{-1}(y') \leq 1$ for any point $y' \in Y'$; since $\psi$ is of fiber type, it is a conic bundle. 
%Let $\cM$ be a minimal rational component which contains general fibers of $\psi$; since our family $\cM$ is unsplit by Proposition~\ref{prop:unsplit:fibration}, 
By our assumption $\tau_X\geq n-2$, the anticanonical degree of any rational curve is at least two; this implies that $\psi$ is a $\P^1$-bundle. Applying \cite[Proposition~2.2]{Fuj12} and its proof, $Y'$ is an $(n-1)$-dimensional smooth Fano variety of $\rho_{Y'}=1$ whose Fano index $i_{Y'}$ satisfies with $i_{Y'}>s>0$; moreover, denoting by $\cO_{Y'}(1)$ the ample generator of ${\rm Pic}(Y')$, $X$ is isomorphic to $\P(\cO_{Y'}\oplus \cO_{Y'}(s))$ and $E$ gives a section of $\psi: X \to Y'$. 

Let $f$ be a minimal extremal rational curve of the ray $R$; by our assumption $\tau_X\geq n-2$ and \cite{CMSB}, the curve $f$ satisfies with $-K_X\cdot f \geq n-2$ and $n \geq -K_E\cdot f$. We have $K_X=\varphi^{\ast}K_Y+\alpha E$ for some $\alpha\in \Q$; then $-\alpha E=-K_X+ \varphi^{\ast}K_Y$ is $\varphi$-nef. Since $\varphi_{\ast}(-\alpha E)=0$ is effective, the negativity lemma (\cite[Lemma~3.39]{KM}) yields that $-\alpha E$ is effective; then it turns out that $E\cdot f$ is negative. Thus we have 
$$
n \geq -K_E\cdot f=-K_X\cdot f-E\cdot f\geq (n-2)+1=n-1.
$$
This implies that $(-K_E\cdot f, -K_X\cdot f, -E\cdot f)=(n, n-1, 1), (n-1, n-2, 1)$ or $(n, n-2,2)$. 

Let us assume that $(-K_E\cdot f, -K_X\cdot f, -E\cdot f)=(n, n-1, 1)$. Since $Y'$ is an $(n-1)$-dimensional smooth Fano variety of $\rho_{Y'}=1$, we see that $i_{Y'}=n$ and $s=1$; by Kobayashi-Ochiai Theorem~\cite{KO}, we obtain $E\cong Y'\cong \P^{n-1}$ and $X\cong \P(\cO_{\P^{n-1}}\oplus \cO_{\P^{n-1}}(-1))$. If $(-K_E\cdot f, -K_X\cdot f, -E\cdot f)=(n-1, n-2, 1)$, the same argument implies that $X\cong \P(\cO_{Q^{n-1}}\oplus \cO_{Q^{n-1}}(-1))$. 
Thus assume that $(-K_E\cdot f, -K_X\cdot f, -E\cdot f)=(n, n-2,2)$. 
Applying \cite{CMSB}, we see that $E\cong Y' \cong \P^{n-1}$ and $X\cong \P(\cO_{\P^{n-1}}\oplus \cO_{\P^{n-1}}(-2))$.
 \end{proof}

\begin{proposition}\label{prop:bir:div:curve} \rm Assume that $\psi$ is of fiber type. If $(\dim E, \dim W)=(n-1, 1)$, then $X$ is isomorphic to one of the following:
\begin{enumerate}
\item the blow-up of $\P^n$ along a line;
\item the blow-up of $Q^4$ along a line;
\item the blow-up of $Q^4$ along a conic which is not on a plane contained in $Q^4$;
%\item the blow-up of $Q^n$ along a conic (intersection of $n-1$ hyperplane sections);
%\item the blow-up of a smooth del Pezzo variety along a complete intersection of $n-1$ hyperplane sections;
%\item the blow up of $\P^4$ along a complete intersection of $n-1$ three quadrics in $\P^4$.
\item $\P^1\times \P(\cO_{\P^{n-2}}\oplus\cO_{\P^{n-2}}(1))$.
\end{enumerate}
\end{proposition}

\begin{proof} By Lemma~\ref{lem:bir:type}, we have $\ell(R)=n-2$; thus \cite[Theorem~5.1]{AO02} implies that $\varphi: X \to Y$ is the blow-up of $Y$ along a smooth curve $W\subset Y$; by our assumption $\tau_X\geq n-2$ and \cite[II~Lemma~3.13]{Kb}, we have $\iota_X\geq 2$. 
Since any fiber $F\cong \P^{n-2}$ of $E=\P(N_{W/Y}^{\vee}) \to W$ is not contracted by $\psi$, $\psi|_{F}: F\cong \P^{n-2}\to \psi(F)$ is a finite morphism. Thus it turns out that $n-2 =\dim \psi(F)\leq \dim Y'<n$. Then our assertion follows from \cite[Lemma~2, Proposition~3 and Proposition~4]{Tsu12},  Example~\ref{eg:Q:blup:l} and Example~\ref{eg:Q:blup:conic}. Remark that the pseudo indices of the Fano varieties $3, 4$, and $5$ in \cite[Proposition~4]{Tsu12} are all one. Thus, our assertion holds.
\end{proof}

\begin{proposition}\label{prop:E-negative:ray} \rm The contraction $\psi: X\to Y'$ of the extremal ray $R'$ is of fiber type.  
\end{proposition}

\begin{proof} First assume that there exists an extremal ray $R''$ such that $E\cdot R''>0$ and the contraction associated to $R''$ is of fiber type. Then Proposition~\ref{prop:bir:fiber:n-1:0} and \ref{prop:bir:div:curve}
 imply that there is no birational contraction of an extremal ray other than $\varphi$. Hence, our assertion holds. Consequently, we may assume that $\psi$ is of birational type. We denote by $E'$ the exceptional divisor of $\psi$ and by $F'$ a 
fiber of $\psi|_{E'}: E'\to W':=\psi(E')$. By the same argument as in Lemma~\ref{lem:bir:type}, we see that $(\dim E', \dim \psi(E))=(n-1, 0)$ or $(n-1,1)$. By our assumption, $E\cdot R'>0$; this implies that $E\cap E'\neq \emptyset$. Then we may find a fiber $F_0$ of $\varphi|_{E}: E\to W$ and a fiber $F_0'$ of $\psi|_{E'}: E'\to W'
$ such that $F_0\cap F_0'\neq \emptyset$. Then we have 
$$
\dim F_0\cap F_0' \geq \dim F_0+\dim F_0'-\dim X.
$$ Remark that  $\dim F_0$ and $\dim F_0'$ are $n-1$ or $n-2$. If $\dim F_0$ or $\dim F_0'$ is $n-1$, then the above inequality tells us that $\dim F_0 \cap F_0'\geq 1$, because we assume that $n\geq 4$. This contradicts to $\varphi\neq \psi$. This yields that $\dim F_0=\dim F_0'=n-2$. The same argument shows that the only possible case is $(n, \dim F_0, \dim F_0')=(4, 2, 2)$. We assume this from now on. By Lemma~\ref{lem:bir:type}, we have $\ell(R')=2$; thus \cite[Theorem~5.1]{AO02} implies that $\psi: X \to Y'$ is the blow-up of $Y'$ along a smooth curve $W'\subset Y'$. Then, by \cite[Theorem~1]{Tsu10-2}, $X$ is isomorphic to the blow-up of $\P^4$ along a smooth, complete intersection of a hyperplane and two quadrics. Then, another contraction of $X$ is the blow-up of a singular variety (see \cite[2~An~example]{Tsu10-2}). This contradicts the fact that $Y$ and $Y'$ are smooth.%\footnote{\cite[2~An~example, Lemma~1]{tsu10-2}により片方のextremal rational curve の$-K_X$-degが1となることも分かる}   
\end{proof}

As a consequence, we obtain Theorem~\ref{them:2}.

\begin{remark}\label{rem:Yasutake}\rm In the paper \cite{Yas14}, without any argument as Proposition~\ref{prop:E-negative:ray}, Yasutake claim that $\psi$ is of fiber type. This part of his proof needs to be clarified.
\end{remark} 

\if0
By the same argument as above, we may obtain a classification of Fano fourfolds of $\iota_X \geq 2$ and $\rho_X\geq 2$:  
\begin{theorem}[\cite{Yas14}]\label{them:bir:contr} \rm  Let $X$ be a smooth Fano fourfold of $\iota_X \geq 2$ and $\rho_X\geq 2$. If $X$ admits a birational elementary contraction, then $X$ is one of the following:
\begin{enumerate}
\item $\P(\cO_{\P^{3}}\oplus \cO_{\P^{3}}(-1))$;
\item $\P(\cO_{Q^{3}}\oplus \cO_{Q^{3}}(-1))$;
\item $\P(\cO_{\P^{3}}\oplus \cO_{\P^{3}}(-2))$;
\item the blow-up of $\P^4$ along a line;
\item the blow-up of $Q^4$ along a line;
\item the blow-up of $Q^4$ along a conic which is not on a plane contained in $Q^4$;
\item $\P^1\times \P(\cO_{\P^{2}}\oplus\cO_{\P^{2}}(-1))$.
\end{enumerate}
\end{theorem}
 \fi
 
\section{Fano varieties of $\tau_X\geq \dim X-2$ and $\rho_X\geq 2$} 
As explained in the introduction, Theorem~\ref{them:fano:collect}~(iii), (iv), (v) imply the following.
\begin{theorem}\label{them:3''} \rm  Let $X$ be a smooth Fano variety of dimension $n \geq 5$. If $\iota_X\geq n-2$ and $\rho_X\geq 2$, then $X$ is isomorphic to one of the following:
\begin{enumerate}
\item $\P^3\times \P^3$;
\item $ \P(\cO_{\P^{3}}^{\oplus 2}\oplus\cO_{\P^{3}}(1))$;
\item $\P^3\times \P^2$;
\item $Q^3\times \P^2$; or
\item $\P(T_{\P^3})$.
\end{enumerate}
\end{theorem}

In this section, we give a complete classification of smooth Fano fourfolds of $\tau_X\geq 2$ and $\rho_X\geq 2$:
\begin{theorem}\label{them:3'} \rm  Let $X$ be a smooth Fano fourfold. If $\iota_X\geq 2$ and $\rho_X\geq 2$, then $X$ is isomorphic to one of the following:
\begin{enumerate}
\item $\P(\cO_{\P^{3}}\oplus \cO_{\P^{3}}(1))$; %In this case, $\wedge^2T_X$ is nef.
\item $\P(\cO_{\P^{2}}^{\oplus 2}\oplus\cO_{\P^{2}}(1))$;;
\item a divisor on $\P^2\times \P^3$ of bidegree $(1, 1)$;
\item $\P^2\times \P^2$;
\item {$\P^1\times Q^3$; or}
\item a Fano fourfold of $i_X=2$.
\end{enumerate}
\end{theorem}

Let us first remark the following:

\begin{lemma}\label{lem:tau:iota:4}\rm For a smooth Fano fourfold, $\tau_X\geq 2$ if and only if $\iota_X\geq 2$. 
\end{lemma}
\begin{proof} The “if” part is trivial, and the “only if” part is immediate from \cite[II~Lemma~3.13]{Kb}.
\end{proof} 

According to Theorem~\ref{them:2}, we only need to deal with the case when any elementary contraction is of fiber type. Moreover if $X$ admits a contraction onto $\P^1$, Proposition~\ref{prop:charact:P1:product} implies that $X$ is isomorphic to the product of $\P^1$ and a Fano threefold $Z$ of $\iota_Z\geq 2$. Thus, throughout this section, we always assume the following:

\begin{assumption}\label{ass:4-fold:tau:2} \rm Let $X$ be a smooth Fano fourfold of $\tau_X\geq 2$ and $\rho_X\geq 2$. Assume that any elementary contraction of $X$ is of fiber type and $X$ does not admit a contraction onto $\P^1$. %Moreover we also assume that $\wedge^2 T_X$ is not nef.
\end{assumption}

\begin{proposition}\label{prop:cont:class}\rm Under the Assumption~\ref{ass:4-fold:tau:2}, one of the following holds:  
\begin{enumerate}
\item $X$ admits a $\P^d$-bundle structure ($d=1, 2$);
\item $X$ admits two $\P^{1}$-fibrations whose special fibers are two-dimensional; 
\item $X$ admits an $\P^{1}$-fibration whose special fibers are two-dimensional and an equidimensional $Q^{2}$-fibration; 
\item $X$ admits two equidimensional $Q^{2}$-fibrations.
\end{enumerate}  
\end{proposition}

\begin{proof} Let us first assume that a smooth elementary contraction $\varphi: X \to Y$ exists. Let $F$ be a fiber of $\varphi$; by \cite[Example~3.8]{casa08}, we see that the the Picard number $\rho_F$ is one; thus $F$ is a Fano variety of $\rho_F=1$. If $\dim F=1$ or $2$, then $\varphi: X \to Y$ is a $\P^d$-bundle ($d=1, 2$). So we assume that $\dim F=3$; in this case, $Y$ is isomorphic to $\P^1$; this is a contradiction.
So, we may assume that any elementary contraction is not smooth. Then, our assertion follows from Theorem~\ref{them:1}. 
\end{proof}

\if0
\begin{lemma}\label{lem:4-fold:rho4} \rm Under the assumption of~\ref{ass:4-fold:tau:2}, $\rho_X\leq 4$, with equality only if $X$ is isomorphic to $\P^1\times \P^1\times \P^1\times \P^1$.   
\end{lemma} 

\begin{proof} The desired inequality follows from \cite[Theorem~2.2]{Wis91}; if the equality holds, then any elementary contraction of $X$ is a $\P^1$-bundle. By \cite[Theorem~1.2]{OSWW}, $X$ is isomorphic to $\P^1\times \P^1\times \P^1\times \P^1$ (see also \cite[Theorem~A.1]{OSWi}).
\end{proof}
\fi

\begin{example}\label{eg:null} Let $\cN$ be the null correlation bundle on $\P^3$, which is defined as the kernel of a surjective map $T_{\P^3}(-1)\to \cO_{\P^3}(1)$ (see for instance \cite[Section~4.2]{OSS}). Let $\cS$ be the spinor bundle on $Q^3$, which is, by definition, the restriction to $Q^3$ of the universal quotient bundle of the Grassmannian $G(1, \P^3)$ (see for instance \cite{Ott88}). Then $\P(\cN)$ is isomorphic to $\P(\cS)$ and we have the following diagram:
\[
  \xymatrix{
    & \P(\cN)\cong \P(\cS)\ar[dl]_{p_1} \ar[dr]^{p_2} &  \\
   \P^3  & & Q^3 }
\]
The variety $\P(\cN)\cong \P(\cS)$ is a Fano fourfold of index two. 
\end{example}

\begin{example}\label{eg:null:pull} We continue to follow the notation in Example~\ref{eg:null}. If we take a suitable double cover $\pi: Q^3\to \P^3$ whose branch divisor $B$ is a quadric, then $X=\P(\pi^{\ast}\cN)$ is a Fano fourfold of index one (see \cite[Corollary]{SSW}). Remark that $X\cong Q^3\times_{\P^3}\P(\cN)$ satisfying the following commutative diagram:
\[
  \xymatrix{
    & X\cong Q^3\times_{\P^3}\P(\cN)\ar[dl]_{q_1} \ar[dr]^{q_2} \ar[dd]_{f} &  \\
   Q^3 \ar[dr]_{\pi}  & &\P(\cN) \ar[dl]^{p_1}\\
       & \P^3 &  }
\]
where $q_1: X\to Q^3$ and $q_2: X\to \P(\cN)$ are natural projections and $f$ is the composition $ \pi\circ q_1$. 
 We prove that $\iota_X=1$. To prove this, let $\ell$ be a line on $Q^3$ such that $\pi(\ell)\subset \P^3$ is also a line. By replacing $\ell$ if necessary, we may assume that there exists a fiber $\ell_2$ of $p_2$ which is isomorphic to $\pi(\ell)$ via $p_1$. For a rational curve $\ell_X:=\ell\times_{\P^3}\ell_2\hookrightarrow X$, 
\begin{eqnarray}
-K_X\cdot \ell_X &=&f^{\ast}K_{\P^3}\cdot \ell_X+{q_1}^{\ast}(-K_{Q^3})\cdot \ell_X+{q_2}^{\ast}(-K_{\P(\cN)})\cdot \ell_X \nonumber \\
&=&K_{\P^3}\cdot \pi(\ell)+(-K_{Q^3})\cdot \ell+(-K_{\P(\cN)})\cdot \ell_2 \nonumber \\
&=& -4+3+2=1 \nonumber 
\end{eqnarray} 
Thus, we see that $\iota_X=1$.
\end{example}

\begin{proposition}\label{prop:P1bundle} \rm Under the assumption of~\ref{ass:4-fold:tau:2}, assume that $X$ admits a $\P^1$-bundle structure $\varphi: X\to Y$. Then $X$ is isomorphic to $\P(\cN)$, where $\cN$ is the null correlation bundle over $\P^3$.
\end{proposition}

\begin{proof} Let us first claim that $\rho_X=2$. To prove this, assume the contrary, that is, $\rho_X\geq 3$. Then it follows from \cite[Corollary~2.9]{KMM92} that $Y$ is a smooth Fano threefold of $\rho_Y\geq 2$. Applying \cite[Lemma~4.2, Lemma~4.4 and Theorem~6.3]{Dru16} (remark that $\P^1\times \P^1\times \P^1$ is missing in the list of \cite[Theorem~6.3]{Dru16}), we see that $Y$ is one of the following:
\begin{enumerate}
\renewcommand{\labelenumi}{(\arabic{enumi})}
\item A double cover of $W=\P(T_{\P^2})$ whose branch locus is a divisor in $|-K_W|$.
\item A divisor on $\P^2\times \P^2$ of bidegree $(1, 1)$, which is isomorphic to $\P(T_{\P^2})$.
\item A divisor on $\P^2\times \P^2$ of bidegree $(1, 2)$.
\item A divisor on $\P^2\times \P^2$ of bidegree $(2, 2)$.
\end{enumerate} 
On the other hand, it follows from \cite[Lemma~2.5]{BCDD03} that $\iota_Y \geq \iota_X \geq2$. Since the pseudoindex of the Fano variety appearing in $(1)$, $(3)$, and $(4)$ is one (for variety in $(1)$, see \cite[p. 53, $n^o$~6~(6.B)]{Matsu95}), we are in the case $(2)$, that is, $Y$ is isomorphic to $\P(T_{\P^2})$. Let $Y$ be $\P(T_{\P^2})$. Since $Y$ is rational, its Brauer group vanishes. This yields that there exists a rank two vector bundle $\cE$ on $Y$ such that $X=\P(\cE)$ and $\varphi$ is the natural projection $X=\P(\cE)\to Y$. Remark that $\iota_X = \iota_Y =2$. Since $\P(T_{\P^2})$ is isomorphic to a divisor on $\P^2\times \P^2$ of bidegree $(1, 1)$, there exist two projections $p_i: \P(T_{\P^2}) \to \P^2$ for $i=1, 2$ as follows:
\[
  \xymatrix{
    & \P(T_{\P^2})\ar[dl]_{p_1} \ar[dr]^{p_2} &  \\
   \P^2  & & \P^2 }
\]
In this diagram, the projective plane $\P^2$ on the left can be seen as a unique family of lines on the projective plane $\P^2$ on the right; then, this diagram is nothing but the universal family of lines. 
 For any fiber $\ell_i$ of $p_i$, \cite[Lemma~2.5]{BCDD03} tells us that $\cE|_{\ell_i}\cong \cO_{\P^1}(a_i)^{\oplus 2}$ for some $a_i \in \Z$ $(i=1,2)$. Tensoring a line bundle, we may assume that $\cE|_{\ell_i}\cong \cO_{\P^1}^{\oplus 2}$ for any fiber $\ell_i$ of $p_i$. Then it follows from Grauert's theorem and Nakayama's lemma that ${p_2}_{\ast}\cE$ is a rank $2$ vector bundle on $\P^2$ and ${p_2}^{\ast}({p_2}_{\ast}\cE)\cong \cE$. Moreover, we have 
 $$
 ({p_2}_{\ast}\cE)|_{p_2(\ell_1)}\cong {p_2}^{\ast}({p_2}_{\ast}\cE)|_{\ell_1}\cong \cE|_{\ell_1}\cong \cO_{\P^1}^{\oplus 2}.
 $$
 This means that ${p_2}_{\ast}\cE$ is a uniform vector bundle on $\P^2$; then \cite{VdV72} implies that ${p_2}_{\ast}\cE\cong \cO_{\P^2}^{\oplus 2}$. As a consequence, $X$ is isomorphic to $\P^1\times \P(T_{\P^2})$, which contradicts to our assumption. Thus, we see that $\rho_X=2$.  
 
Since $Y$ is a Fano threefold of $\rho_Y=1$ and $\iota_Y\geq 2$, \cite[Lemma~2.16]{NO07} implies that $Y$ is $\P^3$ or $Q^3$. Since $Y$ is rational, there exists a rank two vector bundle $\cE$ on $Y$ such that $X=\P(\cE)$. From the classification of rank two Fano bundles over $\P^3$ \cite[Theorem~2.1]{SW90} and that of $Q^3$ \cite[Theorem]{SSW91}, $X$ is isomorphic to one of the following:
\begin{enumerate}  
\item $\P(\cN)$, where $\cN$ is the null correlation bundle over $\P^3$;
\item $\P(\cS)$, where $\cS$ is the spinor bundle over $Q^3$;
\item $\P({\pi}^{\ast}\cN)$, where $\cN$ is the null correlation bundle over $\P^3$ and $\pi: Q^3\to \P^3$ is a double cover.
\end{enumerate}  
By Example~\ref{eg:null} and Example~\ref{eg:null:pull}, $X$ is isomorphic to $\P(\cN)\cong \P(\cS)$.
\end{proof}

\begin{proposition}\label{prop:P2bundle} \rm Under the assumption of~\ref{ass:4-fold:tau:2}, assume that $X$ admits a $\P^2$-bundle structure $\varphi: X\to Y$. Then $X$ is isomorphic to one of the following:
\begin{enumerate}
\item $\P^2\times \P^2$;
\item a divisor on $\P^2\times \P^3$ of bidegree $(1, 1)$;%$\P(T_{\P^2}(-1)\oplus \cO_{\P^2})$, which is 
%\item a divisor on $\P^2\times \P^3$ of bidegree $(2, 1)$.
\end{enumerate}
\end{proposition}
 
\begin{proof} By \cite[Corollary~2.9]{KMM92} and \cite[Lemma~4.2, Lemma~4.4 and Proposition~6.1]{Dru16}, we see that $Y$ is isomorphic to $\P^2$ or $\P^1\times \P^1$. In the latter case, $X$ admits a contraction onto $\P^1$, which contradicts our assumption. Thus $Y$ is isomorphic to $\P^2$; our assertion follows from \cite{SW1}.
\end{proof} 
 
\if0
\begin{lemma}\label{lem:4-fold:P1} \rm Under the assumption of~\ref{ass:4-fold:tau:2}, if $X$ admits a contraction onto $\P^1$, then $X$ is isomorphic to the product of $\P^1$ and a Fano threefold $Z$ of index $2$ and $\rho_Z=1$ except for those of degree $1$.   
\end{lemma} 
 
\begin{proof} By Proposition~\ref{prop:charact:P1:product}, $X$ is isomorphic to the product of $\P^1$ and a threefold $Z$. According to Theorem~\ref{them:nef:2nd}, it is enough to prove that $\wedge^2 T_Z$ is nef and $T_Z$ is not nef. To prove this, let $p_1: X \to \P^1$ and $p_2: X\to Z$ be natural projections. Since we have $T_X \cong p_1^{\ast}T_{\P^1} \oplus p_2^{\ast}T_{Z}$, we obtain 
$$
\wedge^2T_X \cong \left[p_1^{\ast}T_{\P^1} \otimes  p_2^{\ast}T_{Z}\right]\oplus \left[\wedge^2p_2^{\ast}T_{Z}\right], 
$$
$$
\wedge^3T_X \cong \left[p_1^{\ast}T_{\P^1} \otimes \wedge^2 p_2^{\ast}T_{Z}\right]\oplus \left[\wedge^3 p_2^{\ast}T_{Z}\right].
$$
This implies that $Z$ is a Fano threefold with nef $\wedge^2 T_{Z}$ and non-nef $T_Z$.     
\end{proof}
\fi

\begin{proposition}\label{prop:4fold:P1fib} \rm Under the assumption of~\ref{ass:4-fold:tau:2}, assume that $X$ satisfies one of the following:
\begin{enumerate}
\item $X$ admits two $\P^{1}$-fibrations whose special fibers are two-dimensional;
\item $X$ admits a $\P^{1}$-fibration whose special fibers are two-dimensional and an equidimensional $Q^{2}$-fibration.
\end{enumerate}  
Then $\rho_X=i_X=2$.
\end{proposition}

\begin{proof} In both cases, there exists a $\P^{1}$-fibration whose special fibers are two-dimensional, which is denoted by $\varphi: X\to Y$. Remark that the structure of $\varphi$ is described in \cite{Kachi97}. By \cite[Lemma~5.2]{Kachi97}, there exist only finitely many $2$-dimensional fibers of $\varphi$, which are denoted by $F_1=\varphi^{-1}(y_1), \ldots, F_m=\varphi^{-1}(y_m)$. By \cite[Theorem~0.6]{Kachi97} (see also \cite[Lemma~5.9.4]{AndWis98}), $Y$ is smooth and we have $F_i\cong \P^2$ and $N_{F_i/X}\cong \Omega_{\P^2}(1)$. Let $\pi: \tilde{X}\to X$ be the blow-up of $X$ along $F_1 \sqcup \ldots \sqcup F_m$ and $\tilde{Y}$ the blow-up of $Y$ along $\{y_1, \ldots, y_m\}$. Then the exceptional divisor $E$ of $\tilde{X}\to X$ is 
$$\P(N_{F_1/X}^{\vee}) \sqcup \ldots \sqcup \P(N_{F_m/X}^{\vee})\cong\P(T_{\P^2}(-1)) \sqcup \ldots \sqcup \P(T_{\P^2}(-1))
$$ By the universal property of blow-up (see for instance \cite[Proposition~7.14, Corollary~7.15]{Har}), there exists a morphism $\tilde{\varphi}: \tilde{X}\to \tilde{Y}$
\[
  \xymatrix{
   \tilde{X} \ar[r]^{\pi} \ar[d]_{\tilde{\varphi}} &  X\ar[d]^{\varphi}  \\
   \tilde{Y}  \ar[r]&  Y }
\]
making a commutative diagram as shown. Remark that $\tilde{\varphi}: \tilde{X}\to \tilde{Y}$ is a $\P^1$-bundle (cf. \cite[Remark~4.13]{AW93}). Each irreducible component $\P(N_{F_i/X})\cong \P(T_{\P^2}(-1))$ of the exceptional divisor $E$ is isomorphic to a hyperplane section of $\P^2\times \P^2$; the restrictions $\tilde{\varphi}|_{\P(N_{F_i/X}^{\vee})}$ and $\pi|_{\P(N_{F_i/X}^{\vee})}$ give two natural projections of $\P(T_{\P^2}(-1)) \subset \P^2\times \P^2$ onto $\P^2$.

By our assumption, we have another elementary contraction $\psi: X \to Z$ which is 
\begin{itemize}
\item a $\P^{1}$-fibration whose special fibers are two-dimensional; or
\item an equidimensional $Q^{2}$-fibration.
\end{itemize} 
In each case, $Z$ is a smooth variety of dimension at most three (see \cite[Theorem~0.6]{Kachi97} and \cite[Proposition~1.4.1]{AW92}). We proceed as the proof of \cite[Theorem~4.1]{casa08}.
Now we have the following diagram:
\[
  \xymatrix{
   \tilde{X} \ar[r]^{\psi\circ \pi} \ar[d]_{\tilde{\varphi}}  & Z \\
   \tilde{Y}  &   }
\]
This yields a morphism $\tilde{Y}\to {\rm Chow}(Z)$; we denote by $Y'$ the image of this morphism. Let us consider a geometric quotient of $Z$ for $Y'$; by applying \cite[Corollary~1]{BCD07}, there exists an elementary contraction $f: Z\to W$ which contracts a rational curve $\psi\circ \pi(\tilde{\varphi}^{-1}(\tilde{y}))$ for any closed point $\tilde{y} \in \tilde{Y}$. This yields that the surface $\psi(F_i)$ is contracted by $f$. Since $f$ is constructed as a geometric quotient for $Y'$, it is of fiber type; if the relative dimension of $f$ is one, then $f$ is a conic bundle, which contradicts the fact that $f$ contracts a surface. This yields that $f$ has relative dimension at least two. However, if the relative dimension of $f$ is two, then $W$ is isomorphic to $\P^1$; it turns out that $X$ admits a contraction $f \circ\psi: X \to \P^1$; this contradicts our assumption. Consequently, we see that $W$ is a point and $\rho_Z=1$. This implies that $\rho_X=2$. Since $Z$ is a uniruled variety of $\rho_Z=1$, it is a Fano variety. Since $\dim Z=2$ or $3$, $Z$ is $\P^2$ or a Fano threefold of $\rho_Z =1$. Then there exists a rational curve $\ell \subset Z$ such that $-K_Z\cdot \ell$ is equal to the Fano index of $Z$ (see, for instance, \cite{Sok79}). %By the classification of Fano $3$-folds, $Z$ is $\P^3$, $Q^3$, a del Pezzo $3$-fold or Fano $3$-fold of index $1$. 
\begin{claim}\label{cl:section} There exists a rational curve $\tilde{\ell} \subset X$ such that $\psi|_{\tilde{\ell}}: \tilde{\ell} \to \ell$ is birational.
\end{claim}
If $\psi$ is a $\P^{1}$-fibration which is not equidimensional, then it is a smooth morphism away from finitely many fibers (see \cite[Lemma~5.2]{Kachi97}). This means that $\psi|_{\psi^{-1}({\ell})}: \psi^{-1}({\ell}) \to \ell$ is smooth over a nonempty open subset of $\ell$. Thus, the existence of a desired rational curve $\tilde{\ell} \subset X$ follows from the same argument as in \cite[Section~4.3]{Wat21}. So assume that $\psi: X\to Z$ is an equidimensional $Q^2$-fibration. In this case, $Z$ is isomorphic to $\P^2$; then, by replacing $\ell$ if necessary, we may assume that $\psi|_{\psi^{-1}({\ell})}: \psi^{-1}({\ell}) \to \ell$ is smooth over a nonempty open subset of $\ell$. Thus, we are done.

By this claim, we may take a family of minimal sections $\cM \subset {\rm RatCurves}^n(X)$ as in \cite[Section~4.3]{Wat21}, which is unsplit. Moreover, we claim the following:
\begin{claim}\label{cl:contracted} Any $\cM$-curve is contracted by $\varphi$. In particular, $\tilde{\ell}$ is contracted by $\varphi$.
\end{claim}
To prove this, assume the contrary; that is, there exists an $\cM$-curve which is not contracted by $\varphi$. Then, any $\cM$-curve is not contracted by $\varphi$. If $\cM$ is a covering family,  it follows from \cite[Corollary~1]{BCD07} that $\varphi: X\to Y$ is the geometric quotient for $\cM$. This contradicts our assumption. Thus $\codim_X{\rm Locus}(\cM)\geq 1$; then by Proposition~\ref{prop:Ion:Wis:2}, for any point $x \in {\rm Locus}(\cM)$, we have 
$$\dim {\rm Locus}(\cM_x) \geq \deg_{(-K_X)}\cM+\codim_X{\rm Locus}(\cM) -1.
$$ 
By our assumption, we have $\deg_{(-K_X)}\cM \geq 2$; combining this with the above inequality, we see that $\dim{\rm Locus}(\cM_x)\geq2$ for any point $x \in {\rm Locus}(\cM)$ and $\codim_X{\rm Locus}(\cM)=1$. Let us put $D:={\rm Locus}(\cM)$. Here we prove that there exists a curve $\ell_{\varphi} \subset D$ (resp. $\ell_{\psi}\subset D$) which is contracted by $\varphi$ (resp. $\psi$). If $\dim D>\dim \varphi(D)$ and $\dim D>\dim \psi(D)$, then we may find such curves. So assume $\dim D=\dim \varphi(D)=\dim \psi(D)$. Since $\varphi$ and $\psi$ have two-dimensional fibers and $D$ is a divisor on $X$, we may find desired curves by dimension count.

\if0
If $\dim{\rm Locus}(\cM_x)=3$ for some $x\in D$, ${\rm Locus}(\cM_x)$ coincides with $D$; however this implies that $D$ is contracted to a point by $\varphi$, since $\ell_{\varphi} \subset D$ and $\dim N_1(D)=1$ by \cite[II~Corollary~4.21]{Kb}; this is a contradiction. Thus we see that $\dim{\rm Locus}(\cM_x)=2$ for any $x\in D$. 
We claim that the curve $\ell_{\varphi}$ is not contained in ${\rm Locus}(\cM_x)$ for any $x\in D$. In fact, if $\ell_{\varphi}$ is contained in ${\rm Locus}(\cM_{x_0})$ for some $x_0\in D$, any $\cM$-curve is contracted by $\varphi$; this is a contradiction.%; then, for any $x\in D$, a surface ${\rm Locus}(\cM_x)$ is contained in a fiber of $\varphi$; however, this contradicts the fact that the number of two-dimensional fibers of $\varphi$ is finite and $\dim D=3$. 
\fi

We claim that the curve $\ell_{\varphi}$ is not contained in ${\rm Locus}(\cM_x)$ for any $x\in D$. In fact, if $\ell_{\varphi}$ is contained in ${\rm Locus}(\cM_{x_0})$ for some $x_0\in D$, it follows from \cite[II~Corollary~4.21]{Kb} that any $\cM$-curve is contracted by $\varphi$; this is a contradiction. 
Thus we see that $\dim{\rm Locus}(\cM_x)=2$ for any $x\in D$ and we obtain
$$
D=\bigcup_{x\in \ell_{\varphi}} {\rm Locus}(\cM_x)%=\bigcup_{x\in \ell_{\psi}} {\rm Locus}(\cN_x)
$$
By \cite[Lemma~3.2, Remark~3.3]{Occ06}, there exist rational numbers $\lambda\geq 0,  \mu$ and $[C]\in \cM$ such that $\ell_{\psi}=\lambda \ell_{\varphi} +\mu C\in N_1(X)$. For any ample divisor $A$ on $Z$, we have
$$
0=\psi^{\ast}(A)\cdot \ell_{\psi}=\lambda \psi^{\ast}(A)\cdot\ell_{\varphi} +\mu \psi^{\ast}(A)\cdot C.
$$
This implies that $\mu \leq 0$. Therefore we obtain 
$$\ell_{\psi} +(-\mu) C=\lambda \ell_{\varphi}\in \R_{\geq 0}[\ell_{\varphi}].
$$ 
Since $\R_{\geq 0}[\ell_{\varphi}]$ is an extremal ray, this yields $\ell_{\psi} \in \R_{\geq 0}[\ell_{\varphi}]$; this is a contradiction. As a consequence, Claim~\ref{cl:contracted} holds.

On the other hand, since $Y$ is a Fano threefold of $\rho_Y=1$, we may find a rational curve $\ell' \subset Y$ such that $-K_Y\cdot \ell'=i_Y$. By the same argument as in Claim~\ref{cl:section} and Claim~\ref{cl:contracted}, we may find a rational curve $\tilde{\ell'}\subset X$, which is contracted by $\psi$. 

Let $H_Y$ and $H_Z$ be the ample generators of $\rm Pic(Y)$ and $\rm Pic(Z)$, respectively. Since $\rho_X=2$,  
then $\{\varphi^{\ast}H_Y, \psi^{\ast}H_Z\}$ is a basis of $N^1(X)$. Thus there exist 
rational numbers $a, b$ such that $-K_X=a\varphi^{\ast}H_Y+b\psi^{\ast}H_Z$. Denoting by $C_{\varphi}$ and $C_{\psi}$ minimal extremal rational curves of the rays associated to $\varphi$ and $\psi$ respectively, we have
$$
2=-K_X\cdot C_{\psi}=aH_Y\cdot \varphi_{\ast}(C_{\psi}).
$$
On the other hand, we also have
$$
2\leq -K_X\cdot \tilde{\ell'}=aH_Y\cdot \varphi_{\ast}(\tilde{\ell'})=a.
$$
These imply that $H_Y\cdot \varphi_{\ast}(C_{\psi})=1$ and $a=2$. By the same way, we also obtain $H_X\cdot \psi_{\ast}(C_{\varphi})=1$ and $b=2$. As a consequence, $-K_X=2(\varphi^{\ast}H_Y+\psi^{\ast}H_Z)$; this yields that $i_X=2$. 
\end{proof}

\begin{proposition}\label{prop:4fold:Q2fib} \rm Under the assumption of~\ref{ass:4-fold:tau:2}, assume that $X$ admits two equidimensional $Q^{2}$-fibrations. Then $\rho_X=i_X=2$.
\end{proposition}

\begin{proof} This can be proved by the same argument as in Proposition~\ref{prop:4fold:P1fib}.
\end{proof}

Summing up, we obtain Theorem~\ref{them:3'}.

\begin{corollary}[{Gachet's Theorem \cite[Theorem~1.3,~Theorem~1.4]{Gac22}}]\label{cor:str:nef}
Let $X$ be a smooth projective variety of dimension $n \geq 4$ and $\rho_X>1$. Then the following holds:
\begin{enumerate}
\item Assume that $\wedge^3T_X$ is strictly nef. Then $X\cong \P^2 \times \P^2$.  
\item Assume that $\wedge^4T_X$ is strictly nef and $n\geq 5$. Then $X$ is isomorphic to one of the following:
$$ \P(\cO_{\P^{3}}^{\oplus 2}\oplus\cO_{\P^{3}}(1));\quad \P^3\times \P^2;\quad Q^3\times \P^2;\quad  \P(T_{\P^3});\quad \P^3\times \P^3.$$
\end{enumerate}
\end{corollary}

\begin{proof} Let $X$ be a smooth projective variety of dimension $n$ and $\rho_X>1$. Assume that $\wedge^kT_X$ is strictly nef for $1\leq k<n$. Then \cite[Theorem~1.2]{LOY19} yields that $X$ is rationally connected. Applying \cite[Proposition~1.4]{Wat}, $X$ is a Fano variety. Then we see that $\iota_X\geq n+2-k$ (see for instance \cite[Lemma~2.1]{Gac22}). Thus our assertion follows from \cite[Theorem~1.4]{Wat24} and Theorem~\ref{them:3}.
\end{proof}

{\bf Acknowledgements.} The author would like to thank Taku Suzuki for carefully reading the first draft and, pointing out typos, and for helpful suggestions. The author would like to thank Professor Jaroslaw A. Wis\'niewski for answering the author's question on Fano fourfolds and sending his paper \cite{Wis90}. He would also like to thank Bruno Dewer for pointing out typos.
A part of this paper was a question posed during the author's talk at the conference ``Young Perspectives on Algebraic Geometry" held at the Chinese Academy of Sciences in Beijing from December 9th to 11th, 2023. The author would like to express his gratitude to all the organizers who hosted the conference, especially Professor Jie Liu and Professor Baohua Fu. Some results in this paper were obtained in Seoul, Korea, where the author stayed to participate in the conference ``Conference on Singularities and Birational Geometry" held from January 22nd to 26th, 2024. The author would like to express his gratitude to Professor Sung Rak Choi, the organizers, and Yonsei University for giving him the opportunity to visit Korea and for their hospitality during his stay.

%{\bf Conflict of Interest.} The author has no conflicts of interest directly relevant to the content of this article.

%{\bf Data availability.} Data sharing not applicable to this article as no data sets were generated or analyzed during the current study.

\bibliographystyle{plain}
\bibliography{biblio}
\end{document}